\documentclass[12pt]{amsart}
\usepackage{graphicx} % Required for inserting images

\usepackage[utf8]{inputenc}
\usepackage[english]{babel}

\usepackage[margin=30mm,paperwidth=216mm, paperheight=303mm]{geometry}
 
 \usepackage{graphicx}
 \usepackage{amsthm}
 \usepackage{amssymb}
 \usepackage{amsfonts}
 \usepackage{mathrsfs}  
 \usepackage{amsmath}
 \usepackage{mathtools}
 \usepackage{bm}
 \usepackage{yfonts}
 \usepackage{xcolor}
 \usepackage[all]{xy}
 \usepackage{enumerate}
 \usepackage{stmaryrd}
 \usepackage{pdfpages}
 \usepackage{float}
\usepackage{caption}
 \usepackage[pdfborderstyle={/S/U/W 1}]{hyperref}

\usepackage[backend=biber,style=alphabetic, maxbibnames=99]{biblatex}
\usepackage{listings}
\theoremstyle{plain}
\newtheorem{thm}{Theorem}[section]
\newtheorem*{thm*}{Theorem}

\newtheorem*{satz*}{Satz}
\newtheorem{prop}[thm]{Proposition}
\newtheorem*{prop*}{Proposition}
\newtheorem{lem}[thm]{Lemma}
\newtheorem*{lem*}{Lemma}
\newtheorem{cor}[thm]{Corollary}
\newtheorem*{cor*}{Corollary}

\theoremstyle{definition}
\newtheorem{defi}[thm]{Definition}
\newtheorem*{defi*}{Definition}

\newtheorem{notation}[thm]{Notation}

\theoremstyle{remark}
\newtheorem{rem}[thm]{Remark}
\newtheorem*{rem*}{Remark}

\newtheorem*{example*}{Example}

\numberwithin{equation}{section}

\newcommand{\QQ}{\mathbb{Q}}
\newcommand{\RR}{\mathbb{R}}
\newcommand{\CC}{\mathbb{C}}
\newcommand{\ZZ}{\mathbb{Z}}
\newcommand{\NN}{\mathbb{N}}

\newcommand{\FF}{\mathbb{F}}

\newcommand\sumst{\mathop{{\sum}^\ast}\limits}

\newcommand{\mods}[1]{\,(\mathrm{mod}\,{#1})}

\renewcommand{\geq}{\geqslant}
\renewcommand{\leq}{\leqslant}

\DeclareMathOperator{\Kl}{Kl}%Kloosterman sum
\DeclareMathOperator{\K}{K} %trace function 
\DeclareMathOperator{\Hyp}{Hyp} %trace function 

\begin{document}

\title[Non-vanishing of central values with angular restrictions]{Non-vanishing of central values of L-functions with angular restrictions}
\author{Filippo Berta and Svenja zur Verth}

\begin{abstract}
We study the angular restrictions for the second moment of toroidal families of $L$-functions using the general theory of trace functions. With the mollification technique we deduce non-vanishing of a positive proportion. Our two main ingredients are classification results of Katz to determine the sheaves at play and a recent result of Fouvry, Kowalski, Michel and Sawin to bound bilinear sums of their trace functions. 
\end{abstract}

\maketitle
\tableofcontents

\section{Introduction}

In this paper, we continue the study of the average of $L$-values at the critical point over \emph{toroidal families} as introduced by Fouvry, Kowalski and Michel in \cite{FKMsecondmom}. In \emph{loc. cit.} the authors prove that for any $a,b\in \ZZ\smallsetminus\{0\}$  there exist constants $\delta>0$, $C(a,b) \neq 0$, so that for any prime number $q$ the following asymptotic formula\footnote{See Theorem \ref{thm second moment k=0} for details.} holds.
 \begin{equation}\label{eq:introduction0moment}
     \frac{1}{q-1}\sum_{\chi\mods{q}}L\left(\chi^a,\tfrac{1}{2}\right)L\left(\chi^b,\tfrac{1}{2}\right) = \delta_{a+b=0}\log q + C(a,b) + O(q^{-\delta}).
 \end{equation}
Using Cauchy-Schwarz inequality twice, the evaluation of the above moment yields (see \emph{loc. cit.}, Remark 1.2(3))
\begin{equation}\label{eq:intro 0 proportion of non vanishing}
    \vert \{\chi \mods{q}\ | \ L(\chi^a,\tfrac{1}{2})L(\chi^b,\tfrac{1}{2})\neq 0\}\vert \gg \frac{q}{(\log q)^4}, 
\end{equation}
where $q$ goes to infinity along the prime numbers.

In the present paper, we improve \eqref{eq:intro 0 proportion of non vanishing} in two ways: We prove a positive proportion for the simultaneous non-vanishing of $L(\chi^a,\tfrac{1}{2})$ and $L(\chi^b,\tfrac{1}{2})$ and we consider subfamilies of characters $\chi$ whose angle lies in a given section of the circle. More precisely for a primitive Dirichlet character $\chi$ we define its angle $\theta(\chi)$ as the unique number in $(-\pi,\pi]$ so that $\varepsilon(\chi) = e^{i\theta(\chi)}$, where $\varepsilon(\chi)$ is the classical normalized Gauss sum of $\chi$. For the trivial character the angle $\theta (\chi)$ is not defined, we hence restrict our summation to primitive Dirichlet characters. Let $I \subset (-\pi,\pi]$ be an open interval, then we study
\begin{equation*}
    \frac{1}{q-1}\sumst_{\substack{\chi \mods{q}\\ \theta(\chi) \in I}}L\left(\chi^a,\tfrac{1}{2}\right)L\left(\chi^b,\tfrac{1}{2}\right)
\end{equation*}
or rather the smooth version 
\begin{equation*}
    \frac{1}{q-1}\sumst_{\substack{\chi \mods{q}}}L\left(\chi^a,\tfrac 12\right)L\left(\chi^b,\tfrac 12\right)\phi(\theta(\chi)),
\end{equation*}
where $\phi \in C_c^{\infty}((-\pi,\pi]).$  
In this paper, we show the following positive proportion of non-vanishing 
\begin{thm}\label{thm:intro nonvanishing}Let $I \subset (-\pi,\pi]$ be a non-empty interval and let $a,b$ be two non-zero integers. For any prime $q$ we define
\[E(a,b;q,I)= \vert{ \{\chi  \mods q\ \text{ non-trivial s.t. } \  L(\chi^a,\tfrac 12)L(\chi^b,\tfrac 12)\neq  0,\  \theta(\chi) \in I  \}}\vert.\]
Then there exists a constant $C(a,b,I)>0$ so that
    \begin{equation}\label{eq:intrononvanishing}
   \frac{1}{q-1}E(a,b;q,I) \geq C(a,b,I) 
\end{equation}
for $q\to \infty$.
\end{thm}
For a more precise statement see Theorem \ref{thm:non-vanishing}. To prove Theorem \ref{thm:intro nonvanishing}, we combine the mollification technique, results of Zacharias \cite{Zacharias2016MollificationOT} on the mollified fourth moments, the evaluation \eqref{eq:introduction0moment} and our main technical result, which is the following.
\begin{thm}\label{thm:introduction1}
    Let $q$ be a sufficiently large prime. Let $a,b$ non-zero integers such that for $|k|=1$ we have $(a,b)$ is neither $(1,1)$ nor $(-1,-1)$. Let $1\leq l_1,l_2\leq q^\alpha$ for some $\alpha>0$, small. There exists an absolute and effective constant $\eta = \eta(a,b,\alpha) > 0$ and a constant $c(l_1,l_2,b)\geq0$ (absolutely bounded in terms of $q^\alpha$) so that for any integer $k\neq 0$ the following asymptotic formulae are true:
    \begin{align*}
        \frac{1}{q-1}&\sumst_{\chi \mods{q}}\chi(l_1^al_2^b)L\left(\chi^a,\tfrac{1}{2}\right)L\left(\chi^b,\tfrac{1}{2}\right)\varepsilon(\chi)^k \\ 
     =&\begin{cases} \frac{\zeta\left(\frac{|b|+1}{2}\right)}{l_1^{a/2}l_2^{b/2}}+ O_k( q^{-\eta}) &\text{if }(a,b,k)=(1,b>1,-1) \text{ or } (-1,b<-1,1) \\
          c(l_1,l_2,b)+ O_k( q^{-\eta}) &\text{if } (a,b,k)=(1,b<0,-1) \text{ or } (-1,b>0,1) \\
          O_k( q^{-\eta}) &\text{else.}
        \end{cases}
    \end{align*}
\end{thm}

\begin{rem}\label{rem: the FE and its consequences}
    This includes all cases except $(a,b,k)=(1,1,-1)$ or $(-1,-1,1)$. 
By the functional equation one has 
\begin{equation*}\label{eq the FE and its consequences}
    \sumst_{\chi\mods q}L(\bar\chi,\tfrac 12)L(\chi^b,\tfrac 12)\varepsilon(\chi)^k=\sumst_{\chi\mods q}\chi(-1)i^{-\kappa(\chi)}L(\chi,\tfrac 12)L(\chi^b,\tfrac 12)\varepsilon(\chi)^{k-1},\end{equation*}
where $\kappa(\chi)\in {0,1}$. Therefore the cases with $a=1,\ k=-1$ or $a=-1,\ k=1$ degenerate to $k=0$. We then apply an adaptation of \eqref{eq:introduction0moment} to obtain the first two cases of Theorem \ref{thm:introduction1}. Thus the main novelty is the general case. 
\end{rem}

\begin{rem}The constants $C(a,b,I)$ in Theorem \ref{thm:intro nonvanishing} and $\eta$ in Theorem \ref{thm:introduction1} are explicitly computable. On the other hand, the dependency on $k$ in the term $O_k(q^{-\eta})$ in Theorem \ref{thm:introduction1} is not explicit. 
This is a consequence of using the algebraic geometric arguments in \cite{FKMSbilinear}. Their implicit constant depends upon a certain \emph{complexity} attached to the geometric data which certainly depends on $k$.\footnote{See \cite{SAWIN_FAREEFK_2023}, Definitions 3.2,6,6.6, for the formal Definition of the complexity. See also \cite{AppliedladicCohom}, Definition 4.3 for a more gentle first definition.} Currently, this dependence is not explicit, but it is expected to be polynomial. 
\end{rem}

\begin{rem}
    The mollification technique has an interesting beneficial side effect. To prove Theorem \ref{thm:intro nonvanishing} we approximate the interval $I$ with a smooth Fourier series and apply Theorem \ref{thm:introduction1}. Since the dependency in $k$ of the error term $O_k(q^{-\eta})$ is not known, we need the mollifier to tame the contribution coming from the tail of the Fourier series. See Remark \ref{rem:side effect mollification technique} for more details.
\end{rem}

Similar questions of non-vanishing of central $L$-values with angular restriction have already been considered in the literature. We describe two selected results qualitatively. Let $q$ be a prime, sufficiently large and $\theta$ an angle. Among other results, in \cite{Hough}, Theorem $4$, Hough proved that there exists a non-principal $\chi \mods{q}$ such that simultaneously the angle of $\chi$ is close to $\theta$ and the central value of $L(\chi,\tfrac{1}{2})$ is large. 
In \cite{Blomeretal} (see in particular Theorem 1.8), among many other results, Blomer, Fouvry, Kowalski, Michel, Mili\'cevi\'c, and Sawin consider the non-vanishing of $L(f\otimes \chi,\tfrac{1}{2})$ for $f$ a fixed cusp form on $\operatorname{GL}_2$. Defining the angle $\theta(f\otimes \chi)$ as the argument of the central $L$-value, they prove a non-zero proportion of non-vanishing for $L(f\otimes \chi,\tfrac{1}{2})$, when $\chi$ is ranging over Dirichlet characters so that $\theta(f\otimes \chi)$ lies in a prescribed angular sector.

We briefly discuss the strategy for the proof of Theorem \ref{thm:introduction1}. After applying standard techniques, like the approximate functional equation, for the study of Dirichlet $L$-functions, we reduce the proof of Theorem \ref{thm:introduction1} to obtaining estimates for bilinear correlation sums of the following shape:

\begin{equation}\label{eq:introbil}
    \frac{1}{(qMN)^{1/2}}\sum_{m\sim M}\sum_{n\sim N}\alpha_m K(m^an^b) \ll_{a,b,k} q^{-\eta},
\end{equation}
where $K(u) \colon \FF_q \to \CC$ is depending on $a$, $b$ and $k$ and of the shape 
\begin{equation*}
\Kl_k( u;q) \coloneqq \frac{1}{q^{(k-1)/2}}\sum_{\substack{ x_1,\dots, x_k \mods q\\  x_1\cdots  x_k \equiv  u \mods q}}e\left(\frac{ x_1+\dots+ x_k}{q}\right) \end{equation*} or
\begin{equation*}
    \K_k^{a,b}( u;q) \coloneqq \frac{1}{q^{(k+1)/2}} \sum_{\substack{ x_1,\dots, x_k,y_1,y_2 \mods q\\ x_1\cdots  x_ky_1^a y_2^b \equiv  u \mods q}}e\left(\frac{ x_1+\dots +  x_k+y_1+y_2}{q}\right).  \end{equation*}
As it is often the case in analytic number theory, this $K(u)$ is the trace function of a complex of $\ell$-adic sheaves. Although hoping for cancellation as in \eqref{eq:introbil} for every trace function is maybe too optimistic, in \cite{FKMSbilinear}, Fouvry, Kowalski, Michel and Sawin prove cancellation for a broad family of such trace functions, that is, trace functions whose attached complex of $\ell$-adic sheaves is \emph{gallant}. Their result plays a key role in our work. The property of being gallant is depending on the monodromy group of the sheaf attached to the trace function. To identify these our second main ingredient is the work of Katz, \cite{Katzbook} who developed the theory of how to identify the monodromy groups. Our paper is structured in the following way. 
\begin{itemize}
    \item 
In Section \ref{S: AFE, dyadic and PV} we begin the treatment of $ \frac{1}{q-1}\sumst_{\chi \mods{q}}\chi(l_1^al_2^b)L\left(\chi^a,\tfrac{1}{2}\right)L\left(\chi^b,\tfrac{1}{2}\right)\varepsilon(\chi)^k$ by applying the approximate functional equation. 
 \item In Section \ref{S: Hasse- Davenport} we relate all of our functions $K_k^{a,b}$ to hypergeometric sums and in particular to $\ell$-adic sheaves that are well studied. 
 \item In Section \ref{s: dyadic partition and poly vinogradov} we apply a smooth partition of unity and the P\'olya-Vinogradov method to reduce the range of summation to dyadic intervals $(M,2M]$ and $(N,2N]$ such that $q^{1/2-\delta}\leq M,N\leq q^{1/2+\delta}$. We furthermore reduce the general case in Theorem \ref{thm:introduction1} to estimating type I sums. 
 \item In Section \ref{S: Kl1} we 
 prove the estimates on the type I sum required in Section \ref{s: dyadic partition and poly vinogradov} in the case where the trace function is the additive character $e_q$.
 \item In Section \ref{s: weakly gkr} we introduce the definition of gallant and the work of \cite{FKMSbilinear}. We then use the classification theory of Katz to show that our trace functions are gallant. A few cases in which the trace function is not gallant require special treatment. For most of them we apply the Poisson summation formula to \eqref{eq:introbil} in one of the variables to change the trace function to a gallant one. This essentially shows the asymptotic formula in the general case in Theorem \ref{thm:introduction1}.
 \item In Section \ref{s: Adaptation second moment} we adapt the proof of \eqref{eq:introduction0moment} presented in \cite{FKMsecondmom} to obtain the two special cases in Theorem \ref{thm:introduction1} and finish its proof. 
 \item In Section \ref{s: non-vanishing} we introduce mollifiers and show that a positive proportion is non-vanishing as stated in Theorem \ref{thm:intro nonvanishing}.

\end{itemize}

\subsection*{Acknowledgments}
The second author was supported by SNSF, grant 200021-197045. 
The authors would like to thank Philippe Michel for countless discussions, feedback and for generously sharing his ideas with us, and \'Etienne Fouvry and Will Sawin for encouragements and useful comments.

\subsection*{Notation}
We write tuples of $a,b$ and $k$ always in the order $(a,b,k)$. As usual we denote $e^{2\pi i x}$ by $e(x)$ and $e\left(\frac{x}{q}\right)$ by $e_q(x)$.  
We denote the dyadic interval $l\in(L,2L]$ by $l\sim L$. 
We denote by $\bar x$ the inverse of $x$ modulo $q$.
We write $\sumst_{a\mods q}$ for the sum over integers $a$ between $1$ and $q$ that are coprime to $q$ and $\sumst_{\chi\mods q}$ for the sum over all primitive characters modulo $q$. We denote the trivial character by $\chi_0$.
We denote the Gauss sum by $$\varepsilon(\chi)=\frac{1}{q^{1/2}}\sum_{c\mods q}\chi(c)e\left(\frac cq\right).$$
For an integer $a \in \ZZ\smallsetminus\{0\}$ and a finite field $F$ we denote by $\bm{\rho}[a]$ the set of characters of $F^{\times}$ of order dividing $a$, that is $$\bm{\rho}[a] = \{\chi \in \widehat{F^{\times}}|\ \chi^a=1\}.$$ Of course, if $a \equiv b \mods{\vert{F}\vert}$, then $\bm{\rho}[a] = \bm{\rho}[b]$. It is also clear that $\bm{\rho}[a]=\bm{\rho}[-a]$.
For an integer $k\geq 1$ we denote by $\bm{\rho}[a]^k$ the multiset with the same elements as $\bm{\rho}[a]$ in which each element has multiplicity $k$ (instead of $1$). We also denote 
$$\bm{\rho}[a,b] = \bm{\rho}[a]\sqcup\bm{\rho}[b]$$ considered as a multiset.

By $f(x) = O(g(x)), f(x) \ll g(x)$ and $g(x) \gg f(x)$ we mean that there exists some constant $C > 0$ such that for all $x$ sufficiently large $|f(x)| \leq C|g(x)|$. We write $f(x)\asymp g(x)$ if $f(x)\ll g(x)$ and $f(x)\gg g(x)$.
By $f(x)=o(g(x))$ we mean that for all constants $C$ and for $x$ sufficiently large we have $|f(x)|<C|g(x)|$.

The symbol $\varepsilon$ denotes an arbitrarily small positive constant. The exact value of the constant may change between occurrences and even between lines. We use $o(1)$ equivalently. 

We call a trace function gallant if it is associated to a gallant sheaf defined in Definition \ref{def weakly grk}.

\section{The approximate functional equation}\label{S: AFE, dyadic and PV}
Assume $q$ to be prime and let $a,b,k,\xi\in \ZZ\smallsetminus\{0\}$ some non-zero integers, with $(ab\xi,q) = 1$.  
In this section we start the evaluation of 
\begin{equation}\label{eq: defi Mabk}
    M_{a,b,k}(\xi;q):=\frac{1}{q-1} \sumst_{\chi\mods q}\chi(\xi)L(\chi^a,\tfrac 12)L(\chi^b,\tfrac 12)\varepsilon(\chi)^k.\end{equation}
Let $k>0$ and $ u\not \equiv 0\mods q$. We denote the Kloosterman sum by 
\begin{equation}\label{eq defi Kl_k}
\Kl_k( u;q) \coloneqq \frac{1}{q^{(k-1)/2}}\sum_{\substack{ x_1,\dots, x_k \mods q\\  x_1\cdots  x_k \equiv  u \mods q}}e\left(\frac{ x_1+\dots+ x_k}{q}\right). \end{equation} Note that for $k=1$ we non-conventionally denote by $\Kl_1$ the additive character $e_q$. For $a,b$ integers we write
\begin{equation}\label{eq defi K_k^ab}
    \K_k^{a,b}( u;q) \coloneqq \frac{1}{q^{(k+1)/2}} \sum_{\substack{ x_1,\dots, x_k,y_1,y_2 \mods q\\ x_1\cdots  x_ky_1^a y_2^b \equiv  u \mods q}}e\left(\frac{ x_1+\dots +  x_k+y_1+y_2}{q}\right).  \end{equation}

\begin{rem}\label{rem xi with l1 and l2}
    When using the mollification method in Section \ref{s: non-vanishing} we choose $\xi=l_1^al_2^b$ for $l_1,l_2\leq L=q^\alpha$ with $\alpha>0$ small. If needed, we write $\xi$ in this more explicit form.  
\end{rem}
Let $X,Y$ be such that $XY=q^2$.
For $k>0$, $u\not\equiv 0 \mods{q}$ and $V$ which satisfies the following properties 
\begin{align}
    V(y) = 1+O(y^A)&,   \label{eq: decaying V 1} \\
    y^jV^{(j)}(y) \ll_{j,A} y^{-A}&, \qquad \forall A>1,j \in \ZZ_{\geq 1}, y>0.\label{eq: decaying V 2} 
\end{align}we define
\begin{align*}M_{a,b,\Kl_k}(u) &= \frac{1}{q^{1/2}}\sum_{\substack{m,n\geq 1\\(mn,q)=1}}\frac{1}{(mn)^{1/2}}V\left(\frac{mn}{X}\right)\Kl_k(um^{-a}n^{-b};q),\\
M_{a,b,\K_k^{a,b}}(u) &= \frac{1}{q^{1/2}}\sum_{\substack{m,n\geq 1\\(mn,q)=1}}\frac{1}{(mn)^{1/2}}V\left(\frac{mn}{Y}\right)\K^{a,b}_k(um^an^b;q).\\
\end{align*}
We do not display the dependency on $V$ in the notation. 
For us $V= V_{a,b,+}$ or $V=V_{a,b,-}$, where $V_{a,b,\pm }$ are chosen as in \cite{IwaKowAnalyticNT}, Theorem 5.3 with $G(u)=1$, more precisely let
\begin{equation*}
    V_{a,b,\chi}(y) = \frac{1}{2\pi i}\int_{\Re(u)=\sigma}y^{-u}\frac{\gamma_{t(\chi^a)}(1/2+u)}{\gamma_{t(\chi^a)}(1/2)}\frac{\gamma_{t(\chi^b)}(1/2+u)}{\gamma_{t(\chi^b)}(1/2)} \ \frac{\mathrm{d}u}{u},
\end{equation*}
    where $t(\chi)=\frac{1-\chi(-1)}{2}$ and $\gamma_j(u ) = \gamma(u+j)$ for $j\in \{0,1\}$. Note that $V_{a,b,\chi}$ only depends on $a,b$ and the parity of the character. We therefore denote  $\gamma_{t(\chi^a)}\gamma_{t(\chi^b)}$ by $\gamma_{(a,b,+)}$ or $\gamma_{(a,b,-)}$ and $V_{a,b,\chi}$ by $V_{a,b,+}$ or $V_{a,b,-}$, respectively.
   The functions $V_{a,b,\pm}$ satisfy the properties \eqref{eq: decaying V 1} and \eqref{eq: decaying V 2}.

\begin{prop}\label{prop: application AFE} Let $q$ be prime and $a,b,k\in \ZZ\smallsetminus\{0\}$. 
Let $$M_{a,b,k}(\xi;q)=\frac{1}{q-1} \sumst_{\chi\mods q}\chi(\xi)L(\chi^a,\tfrac 12)L(\chi^b,\tfrac 12)\varepsilon(\chi)^k.$$ 
    We can write $M_{a,b,k}(\xi;q)$ as a linear combination of $M_{a,b,\Kl_{\vert{k}\vert}}(\pm \bar \xi)$ and $M_{a,b,\K^{a.b}_{\vert{k}\vert}}(\pm \bar \xi)$ with absolutely bounded coefficients and an error of size $O\left(q^{-1/2}+q^{-1}\left(X^{1/2+\varepsilon}+Y^{1/2+\varepsilon}\right)\right)$. 
\end{prop}

\begin{proof} 
We first assume that $k>0$ and start with the analysis for the case where neither $\chi^a$ nor $\chi^b $ is trivial. 
The product $L(\chi^a,s)L(\chi^b,s)$ is again an $L$-function of conductor $q^2$ with root number $i^{-(t(\chi^a)+t(\chi^b))}\varepsilon(\chi^a)\varepsilon(\chi^b)$ and gamma factor $\gamma(\chi^a,s)\gamma(\chi^b,s)$. This $L$-function is entire unless $\chi^a$ or $\chi^b$ is trivial and never has a pole at $s=1/2$.
By the approximate functional equation (see \cite{IwaKowAnalyticNT}, Theorem 5.3) we have, according to the parity of $\chi$,

\begin{equation}\label{eq: AFE}
L(\chi^a,\tfrac 12)L(\chi^b,\tfrac 12)=\sum_{m,n\geq 1}\frac{\chi(m^an^b)}{(mn)^{1/2}}V_{a,b,\pm}\left(\frac{mn}{X}\right)+\frac{\varepsilon(\chi^a)\varepsilon(\chi^b)}{i^{t(\chi^a)+t(\chi^b)}}\sum_{m,n\geq 1}\frac{\bar \chi(m^an^b)}{(mn)^{1/2}}V_{a,b,\pm}\left(\frac{mn}{Y}\right).\end{equation}
 In the case where either $\chi^a$ or $\chi^b $ is trivial, we use the convexity bound as stated in \cite{IwaKowAnalyticNT}, Theorem 5.23 and the fact that $|\varepsilon(\chi)|\leq 1$ irrespective of the character. Hence for all characters $\chi$ we have
  \begin{equation}\label{eq: chia or chi b trivial}
        \chi(\xi)L(\chi^a,\tfrac 12)L(\chi^b,\tfrac 12)\varepsilon(\chi)^k\ll q^{1/2}.
    \end{equation} 
Furthermore 
\begin{align*}
    &\varepsilon(\chi)^k\sum_{m,n\geq 1}\frac{\chi(\xi m^an^b)}{(mn)^{1/2}}V_{a,b,\pm}\left(\frac{mn}{X}\right)\ll X^{1/2+\varepsilon},\\
    \varepsilon(\chi^a)\varepsilon(\chi^b) &\varepsilon(\chi)^k\sum_{m,n\geq 1}\frac{\bar \chi(\bar \xi m^an^b)}{(mn)^{1/2}}V_{a,b,\pm}\left(\frac{mn}{Y}\right)\ll Y^{1/2+\varepsilon}.
\end{align*}
Hence using \eqref{eq: AFE} and the above we now sum over all but the trivial character and obtain

\begin{align*}
M_{a,b,k}(\xi;q)=&\frac{1}{q-1}\sumst_{\chi\mods q}\sum_{m,n\geq 1}\frac{\chi(\xi m^an^b)}{(mn)^{1/2}}\varepsilon(\chi)^kV_{a,b,\pm}\left(\frac{mn}{X}\right)\\&+\frac{1}{q-1}\sumst_{\chi\mods q}\frac{\varepsilon(\chi^a)\varepsilon(\chi^b)\varepsilon(\chi)^k}{i^{t(\chi^a)+t(\chi^b)}}\sum_{m,n\geq 1}\frac{\bar \chi(\bar \xi m^an^b)}{(mn)^{1/2}}V_{a,b,\pm}\left(\frac{mn}{Y}\right)\\ & +O\left(q^{-1/2}+q^{-1}\left(X^{1/2+\varepsilon}+Y^{1/2+\varepsilon}\right)\right).
\end{align*}
Next we split the sum over the characters into the even and the odd parts, $M_{a,b,k}(\xi;q)^{even}$ and respectively $M_{a,b,k}(\xi;q)^{odd}$. Using the indicator functions $\frac{1\pm \chi(-1)}{2}$ and the fact that $t(\chi^a)=0$ for $\chi$ even, we obtain 
\begin{align*}
    M_{a,b,k}(\xi;q)^{even}=&\frac{1}{2(q-1)}\sumst_{\chi\mods q}\sum_{m,n\geq 1}\frac{\chi(\xi m^an^b)}{(mn)^{1/2}}\varepsilon(\chi)^kV_{a,b,+}\left(\frac{mn}{X}\right)\\&
    +\frac{1}{2(q-1)}\sumst_{\chi\mods q}\sum_{m,n\geq 1}\frac{\chi(-\xi m^an^b)}{(mn)^{1/2}}\varepsilon(\chi)^kV_{a,b,+}\left(\frac{mn}{X}\right)\\&
    +\frac{1}{2(q-1)}\sumst_{\chi\mods q}\varepsilon(\chi^a)\varepsilon(\chi^b)\varepsilon(\chi)^k\sum_{m,n\geq 1}\frac{\bar \chi(\bar \xi m^an^b)}{(mn)^{1/2}}V_{a,b,+}\left(\frac{mn}{Y}\right) \\&
    +\frac{1}{2(q-1)}\sumst_{\chi\mods q}\varepsilon(\chi^a)\varepsilon(\chi^b)\varepsilon(\chi)^k\sum_{m,n\geq 1}\frac{\bar \chi(-\bar \xi m^an^b)}{(mn)^{1/2}}V_{a,b,+}\left(\frac{mn}{Y}\right).
\end{align*}
The odd part can be expressed in a similar manner.
For the first summand of the even part this yields, denoting $V=V_{a,b,+}$, 
\begin{align*}
&  \frac{1}{(q-1)}\sumst_{\chi\mods q}\sum_{m,n\geq 1}\frac{\chi(\xi m^an^b)}{(mn)^{1/2}}\varepsilon(\chi)^kV\left(\frac{mn}{X}\right) \\
    & =\frac{1}{(q-1)q^{k/2}}\sumst_{\chi\mods q}\sum_{m,n\geq 1}\frac{V\left(\frac{mn}{X}\right)}{(mn)^{1/2}}\sum_{\substack{ x_1, \ldots,  x_k\mods q}}\chi( \xi m^an^bx_1\cdot \ldots \cdot  x_k)e\left(\frac{ x_1+\ldots+ x_k}{q}\right) \\
    & =\frac{1}{(q-1)q^{k/2}}\sum_{\substack{m,n\geq 1\\(mn,q)=1}}\frac{V\left(\frac{mn}{X}\right)}{(mn)^{1/2}}\sum_{\substack{ x_1, \ldots,  x_k\mods q}}e\left(\frac{ x_1+\ldots+ x_k}{q}\right)\left(-1+\sum_{\chi\mods q}\chi(\ldots)\right)\\
     & =\frac{1}{q^{k/2}}\sum_{\substack{m,n\geq 1\\(mn,q)=1}}\frac{1}{(mn)^{1/2}}V\left(\frac{mn}{X}\right)\sum_{\substack{ x_1,\ldots, x_k \mods q\\ \xi m^an^b x_1\cdot \ldots \cdot  x_k \equiv 1\mods q}}e\left(\frac{ x_1+\ldots+ x_k}{q}\right),
\end{align*}
where in the last step we have used orthogonality of characters for the sum over $\chi$. 
We treat the remaining three summands analogously and obtain

\begin{align*}
M_{a,b,k}^{even}(\xi;q)=\frac{1}{2q^{1/2}}\Bigg(&\sum_{\substack{m,n\geq 1\\(mn,q)=1}}\frac{1}{(mn)^{1/2}}V\left(\frac{mn}{X}\right)\left(\Kl_k(\bar \xi m^{-a}n^{-b};q)+\Kl_k(-\bar \xi m^{-a}n^{-b};q)\right)\\ + &\sum_{\substack{m,n\geq 1,\pm\\(mn,q)=1}}\frac{1}{(mn)^{1/2}}V\left(\frac{mn}{Y}\right)\left(\K^{a,b}_k(\bar \xi m^{a}n^{b};q)+\K^{a,b}_k(- \bar \xi m^{a}n^{b};q)\right)\Bigg). \end{align*}

For the odd part we note that $t(\chi^a)+t(\chi^b)$ depends only on the parity of $a,b$ if the parity of $\chi$ is set to be odd. We have $t(\chi^a)+t(\chi^b)=(a\mods 2)+(b\mods 2)$.  We then analogously, with this time $V=V_{a,b,-}$, have

\begin{align*}
M_{a,b,k}^{odd}(\xi;q)=\frac{(-i)^{\tau}}{2q^{1/2}}\Bigg(&\sum_{\substack{m,n\geq 1\\(mn,q)=1}}\frac{1}{(mn)^{1/2}}V\left(\frac{mn}{X}\right)\left(\Kl_k(\bar \xi m^{-a}n^{-b};q)-\Kl_k(-\bar \xi m^{-a}n^{-b};q)\right)\\ + &\sum_{\substack{m,n\geq 1,\pm\\(mn,q)=1}}\frac{1}{(mn)^{1/2}}V\left(\frac{mn}{Y}\right)\left(\K^{a,b}_k(\bar \xi m^{a}n^{b};q)-\K^{a,b}_k(-\bar \xi m^{a}n^{b};q)\right)\Bigg), \end{align*}
where $\tau=(a\mods 2)+(b\mods 2)$.
Using the definitions of $M_{a,b,\Kl_k}(u)$ and $M_{a,b,\K_k^{a,b}}(u)$ this shows the claim. 
For $k<0$ note that $\varepsilon(\chi)^{-1}=\varepsilon(\bar\chi)\chi(-1)$. Therefore the above analysis is analogous and leads to the same objects for $|k|$.
\end{proof}

\begin{rem}
    If we are only interested in non-vanishing results we could restrict ourselves to $a,b>0$ and $k\in \ZZ\smallsetminus \{0\}$ because $L(\chi,\tfrac 12)$ vanishes if and only if $L(\bar\chi,\tfrac 12)$ does. We do this in Section \ref{s: non-vanishing} but consider general $a$ and $b$ until then for possible future applications.
\end{rem}

\begin{rem}It can be useful to unbalance the functional equation, see for example \cite{FKMsecondmom} or \cite{radziwiłł2023nonvanishingtwiststextgl4mathbbamathbbqlfunctions}.
   In the present work, the choice $X= q^{1+\kappa}$ and $Y=q^{1-\kappa}$ for a small $\kappa>0$ would allow to treat the $\K_k^{a,b}$ trivially. 
  For the treatment of $\Kl_k$ we could hope to adapt the strategy in \cite{KMSkloostermanStratification} to obtain polynomial control in the complexity of the trace functions $\Kl_k$.
However, one has to verify that the strategy in \emph{loc. cit.} runs through with the presence of exponents $a$ and $b$ in the argument of the Kloosterman sum. We henceforth consider the balanced situation and treat $\K_k^{a,b}$ non-trivially.
   \end{rem}

\section{Reduction to hypergeometric sums}\label{S: Hasse- Davenport}
While the Kloosterman sums $\Kl_k$ are classically related to a well studied $\ell$-adic sheaf, this is not the case for $\K_k^{a,b}$. In this section we establish a relation of the latter to a well known $\ell$-adic sheaf using the Hasse-Davenport Relation. This section closely follows upcoming work of Fouvry, Kowalski, Michel and Sawin, \cite{FKMSthirdmom}. 
Recall that
\[\K_k^{a,b}(u;q) \coloneqq \frac{1}{q^{(k+1)/2}} \sum_{\substack{ x_1,\dots, x_k, y_1, y_2 \mods q\\ x_1\cdots  x_k y_1^a  y_2^b \equiv u \mods q}}e\left(\frac{ x_1+\dots +  x_k+ y_1+ y_2}{q}\right).  \]
More generally, for every finite field $F$ and a choice of a non-trivial $\ell$-adic character $\psi$ on $F$ we define 

\begin{equation}\label{eq:4.1}\K_k^{a,b}( u;\psi,F) \coloneqq \frac{1}{\vert{F}\vert^{(k+1)/2}} \sum_{\substack{ x_1,\dots, x_k, y_1, y_2 \in F\\ x_1\cdots  x_k y_1^a  y_2^b =  u}}\psi( x_1+\dots +  x_k+ y_1+ y_2).\end{equation}
For two functions $f,g$ on $F^{\times}$, we define the multiplicative convolution
\[f\ast g(u ) = \frac{1}{{\vert{F}\vert^{\frac{1}{2}}}}\sum_{xy=u}f(x)g(y).\]
We see that
\begin{align*}
&\left(\sum_{x^{a}=u_1}\psi( ax) \ast \sum_{x^{b}=u_2}\psi(bx) \ast \psi \ast \dots \ast \psi\right)(u) \\&\qquad\qquad\qquad= \frac{1}{\vert{F}\vert^{\frac{k+1}{2}}}\sum_{x_1^{a}x_2^{b}y_1\cdots y_k=u }\psi(ax_1+bx_2+y_1+\cdots+ y_k)\\ &\qquad\qquad\qquad=\K_k^{a,b}(a^{a}b^{b}u;\psi,F)
\end{align*}

By the general theory of convolution \cite{Katz2012Convolution} (see also \cite{Katzbook}), we deduce that every $\K_k^{a,b}(u;\psi,F)$ is the trace function of a complex of $\ell$-adic sheaves, mixed of weight $\leq 0$ and lisse on $\mathbb{G}_{m/F}$, which we denote $\mathcal{K}_k^{a,b}(\psi)$.

 To further study the correlation properties of $\K_k^{a,b}(\bullet;\psi)$ we first relate $\mathcal{K}^{a,b}_k(\psi)$ to a well-studied $\ell$-adic sheaf.

Let $F$ be a finite field and let $\bm{\chi}$ and $\bm{\theta}$ denote multisets (possibly empty) of $\ell$-adic multiplicative characters defined on $F^{\times}$ of cardinality $r$ and $t$ respectively. The hypergeometric function attached to $\bm{\chi}$ and $\bm{\theta}$ is the following function on $F^{\times}$:
\begin{equation*}
    \Hyp(u,\bm \chi,\bm \theta;\psi,F)=\frac{1}{|F|^{\frac{r+t-1}2}}\sum_{\substack{x_1,\ldots,x_r,y_1,\ldots,y_t\\\frac{x_1\cdots x_r}{y_1\cdots y_t}=u}}\psi(x_1+\ldots+x_r-y_1-\ldots-y_t)\prod_{i=1}^r\chi_i(x_i)\prod_{i=1}^t\bar\theta_i(y_i).
\end{equation*}

Each function $\operatorname{Hyp}(u;\bm{\chi},\bm{\theta};\psi,F)$ is the trace function of $\mathcal{H}(\bm{\chi},\bm{\theta},\psi)$, the so-called \emph{hypergeometric sheaf}, lisse over ${\mathbb{G}_m}_{/F}$. Hypergeometric sheaves have been extensively studied in \cite{Katzbook}, Chapter 8. 

We set some notation: we denote by $\sqcup$ the union of multisets, in particular if $\rho \in \bm{\chi},\bm{\theta}$ has multiplicity $r_1$ and $r_2$ respectively, then $\rho$ has multiplicity $r_1+r_2$ in $\bm{\chi}\sqcup \bm{\theta}.$ For two multisets $\bm{\chi}$ and $\bm{\theta}$ as above, the intersection $\bm{\chi}\sqcap \bm{\theta}$ is the multiset where $\rho$ appear in $\bm{\chi}\sqcap \bm{\theta}$ with multiplicity equal to the minimum of the multiplicities of $\rho$ in $\bm{\chi}$ and $\bm{\theta}$. The multiset $\bm{\chi}\smallsetminus\bm{\chi}\sqcap\bm{\theta}$ is obtained by subtracting multiplicities, that is if $\rho$ has multiplicity $m$ in $\bm{\chi}$ and multiplicity $n$ in $\bm{\chi}\sqcap \bm{\theta}$, then $\rho$ has multiplicity $m-n$ in $\bm{\chi}\smallsetminus\bm{\chi}\sqcap\bm{\theta}$. We say that $\bm{\chi}$ and $\bm{\theta}$ are distinct if $\bm{\chi}\smallsetminus\bm{\chi}\sqcap\bm{\theta}$ or $\bm{\theta}\smallsetminus\bm{\chi}\sqcap\bm{\theta}$ is non-empty (that is, it has an element of multiplicity strictly bigger than $0$).

The following can be found in \cite{Katz88}, § 5.6. 

\begin{lem}[Hasse-Davenport Relation]\label{lem:Hasse-davenport}
    Let $F$ be a finite field. Let $\psi \colon F\to \CC^{\times}$ and $\chi \colon F^{\times}\to \CC^{\times}$ characters. We denote by $\varepsilon(\psi,\chi) = \frac{1}{\vert{F}\vert^{\frac{1}{2}}}\sum_{x\in F^{\times}}\psi(x)\chi(x)$ the normalized Gauss sum. For $\lambda \in F$ we denote the character $F\ni x \mapsto \psi(\lambda x)$ by $\psi_\lambda$.

Then we have for any $N\in \NN$ so that $N|\vert{F^{\times}}\vert $ that 
\begin{equation*}
    -\varepsilon(\psi_N,\chi^N) = \varepsilon_N(\psi)\prod_{\substack{\rho \in \widehat{F^{\times}}\\\rho^N=1}}\varepsilon(\psi,\chi\rho),
\end{equation*}
where $\varepsilon_N(\psi)$ is a complex number of absolute value $1$. 
\end{lem}

 Recall the definition of $\bm \rho[a]$ from the notation section. We deduce the following

\begin{lem}\label{l: hasse davenport}
Let $a,b\in \ZZ\smallsetminus\{0\}$ and let $F$ be a finite field so that $a,b|\vert{F^{\times}}\vert$. Let $u\neq 0.$
We can express the sum $\K_k^{a,b}$ as a hypergeometric sum in the following way: 

\begin{equation}\label{eq:Kk=hyp}    \K_k^{a,  b}(a^{a}b^{b}u;\psi,F) = \begin{cases}
        \varepsilon_a(\psi)\varepsilon_b(\psi)\operatorname{Hyp}(u;\bm{\rho}[a,b]\sqcup \bm{\rho}[1]^k,\varnothing;\psi,F)&  a,b>0, \\
        \varepsilon_a(\psi)\varepsilon_b(\psi)\operatorname{Hyp}(u;\bm{\rho}[a]\sqcup \bm{\rho}[1]^k,\bm{\rho}[b];\psi,F)& a>0>b, \\
        \varepsilon_a(\psi)\varepsilon_b(\psi)\operatorname{Hyp}(u;\bm{\rho}[b]\sqcup \bm{\rho}[1]^k,\bm{\rho}[a];\psi,F)& b>0>a,\\ \varepsilon_a(\psi)\varepsilon_b(\psi)
        \operatorname{Hyp}(u;\bm{\rho}[1]^k,\bm{\rho}[a]\sqcup \bm{\rho}[b];\psi,F)&0>a,b.\
    \end{cases}\end{equation}
\end{lem}

\begin{proof}
  We orient ourselves at \cite{FKMSthirdmom}. Suppose first $a>0$, we compute the discrete Mellin transform of the following function in two ways 
\begin{equation}\label{eq:Hasse-davenport2}
    \widehat{F^{\times}}\rightarrow \CC, \chi \mapsto -\varepsilon(\psi_a,\chi^a).
\end{equation}
By the definition of discrete Mellin transform we have first 
\begin{align}\label{eq:Hasse-davenport3}
    F^{\times}\ni u\mapsto&\ -\frac{1}{\vert{F^{\times}}\vert^{\frac{1}{2}}}\sum_{\chi \in \widehat{F^{\times}}}\varepsilon(\psi_a,\chi^a)\overline{\chi}(u) \nonumber \\ &=-\frac{1}{\vert{F^{\times}}\vert^{\frac{1}{2}}\vert{F}\vert^{\frac{1}{2}}}\sum_{x\in F^{\times}}\psi(ax)\sum_{\chi \in \widehat{F^{\times}}}\chi(x^a)\overline{\chi}(u) = -\frac{\vert{F^{\times}}\vert^{\frac{1}{2}}}{\vert{F}\vert^{\frac{1}{2}}}\sum_{x\colon x^a=u}\psi(ax).
\end{align}
On the other hand, by Lemma \ref{lem:Hasse-davenport} we have
\begin{align}\label{eq:Hasse-davenport4}
F^{\times}\ni u\mapsto & \frac{1}{\vert{F^{\times}}\vert^{\frac{1}{2}}}\sum_{\chi\in \widehat{F^{\times}}}\varepsilon_a(\psi)\prod_{\substack{\rho \in \widehat{F^{\times}}\\\rho^a=1}}\varepsilon(\psi,\chi\rho)\overline{\chi}(u) \nonumber \\ &= \frac{\varepsilon_a(\psi)}{\vert{F^{\times}}\vert^{\frac{1}{2}}\vert{F}\vert^{\frac{a}{2}}}\sum_{x_1,\dots,x_a}\psi(x_1+\dots+x_a)\rho_1(x_1)\cdots\rho_a(x_a)\sum_{\chi\in \widehat{F^{\times}}}\chi(x_1\cdots x_a)\overline{\chi}(u)\nonumber \\ &=\frac{\vert{F^{\times}}\vert^{\frac{1}{2}}\varepsilon_a(\psi)}{\vert{F}\vert^{\frac{1}{2}}}\Kl(u;\rho[a];\psi,F).
\end{align}
We deduce for $a>0$ that 
\begin{equation}\label{eq: K Kl connection}
    -\sum_{x:x^a=u}\psi(ax) =\varepsilon_a(\psi)\Kl(u;\rho[a];\psi,F).
\end{equation}
Similarly, we compute for $a<0$  
\begin{equation*}
    -\sum_{x:x^{a} =u}\psi(ax) 
    = \varepsilon_a(\psi)\operatorname{Hyp}(u,\varnothing,\rho[a];\psi,F) = \varepsilon_a(\psi)\Kl(u^{-1},\rho[a];\overline{\psi},F).
\end{equation*}
Recall that, 
\begin{align*}
\left(\sum_{x^{a}=u_1}\psi( ax) \ast \sum_{x^{b}=u_2}\psi(bx) \ast \psi \ast \dots \ast \psi\right)(u) 
=\K_k^{a,b}(a^{a}b^{b}u;\psi,F)
\end{align*}
and similarly
\begin{align*}
    &\left(\varepsilon_a(\psi)\Kl(\bullet^{\operatorname{sgn}(a) };\rho[a];\psi^{\operatorname{sgn}(a)},F)*\varepsilon_b(\psi)\Kl(\bullet^{\operatorname{sgn}(b)};\rho[b];\psi^{\operatorname{sgn}(b)},F) * \psi* \dots *\psi\right)(u) =\\ 
    \\ &\begin{cases}
        \varepsilon_a(\psi)\varepsilon_b(\psi)\operatorname{Hyp}(u;\bm{\rho} [a,b]\sqcup \bm{\rho} [1]^k,\varnothing;\psi,F)&  a,b>0, \\
        \varepsilon_a(\psi)\varepsilon_b(\psi)\operatorname{Hyp}(u;\bm{\rho}[a]\sqcup \bm{\rho}[1]^k,\rho[b];\psi,F)& a>0>b, \\
        \varepsilon_a(\psi)\varepsilon_b(\psi)\operatorname{Hyp}(u;\bm{\rho}[b]\sqcup \bm{\rho}[1]^k,\bm{\rho}[a];\psi,F)& b>0>a,\ \\\varepsilon_a(\psi)\varepsilon_b(\psi)
        \operatorname{Hyp}(u;\bm{\rho}[1]^k,\bm{\rho}[a]\sqcup \bm{\rho}[b];\psi,F)&0>a,b.
    \end{cases}
\end{align*}
\end{proof}

For any positive integer $n$ let $F_n/\FF_q$ denote the finite extension of degree $n$. We denote by $\psi_n$ the standard additive character of $F_n$, that is $\psi_n( u) = e_q(\operatorname{tr}_{F_n/\FF_q}( u))$ and denote $\K_k^{a,b}( u;\psi_n,F_n) = \K_k^{a,b}( u;F_n)$. Similarly we denote $\operatorname{Hyp}(u,\bm{\chi},\bm{\theta};\psi_n,F_n)= \operatorname{Hyp}(u,\bm{\chi},\bm{\theta};F_n)$, here is implicitly understood that $\bm{\chi}$ and $\bm{\theta}$ are family of characters in $F_n^{\times}$. We denote the attached complexes of sheaves lisse over $\mathbb{G}_{m/\FF_q}$ by $\mathcal{K}_k^{a,b}$ and $\mathcal{H}(\bm{\chi},\bm{\theta},\psi)$, respectively.

\begin{cor}\label{cor: K_k is hypergeometric}Let $u\neq 0.$ For every finite extension $F_n/\FF_q$ in which $X^a-1$, $X^b-1$ split we have the equality of trace functions   \begin{equation*}
    K_k^{a,b}(a^ab^bu;F_n) = \begin{cases}
        \varepsilon_a(\psi_n)\varepsilon_b(\psi_n)\operatorname{Hyp}(u;\bm{\rho} [a,b]\sqcup \bm{\rho} [1]^k,\varnothing;F_n)&  a,b>0, \\
        \varepsilon_a(\psi_n)\varepsilon_b(\psi_n)\operatorname{Hyp}(u;\bm{\rho}[a]\sqcup \bm{\rho}[1]^k,\bm{\rho}[b];F_n)& a>0>b, \\
        \varepsilon_a(\psi_n)\varepsilon_b(\psi_n)\operatorname{Hyp}(u;\bm{\rho}[b]\sqcup \bm{\rho}[1]^k,\bm{\rho}[a];F_n)& b>0>a,\ \\\varepsilon_a(\psi_n)\varepsilon_b(\psi_n)
        \operatorname{Hyp}(u;\bm{\rho}[1]^k,\bm{\rho}[a]\sqcup \bm{\rho}[b];F_n)&0>a,b,
    \end{cases}
\end{equation*}
which implies that $\mathcal{K}_k^{a,b}$ is geometrically isomorphic (in particular, they have same geometric monodromy group and local monodromies) to 
\begin{equation*}
    [\times (a^ab^b)^{-1}]^*\begin{cases}
        \mathcal{H}(\bm{\rho} [a,b]\sqcup \bm{\rho} [1]^k,\varnothing)&  a,b>0, \\
      \mathcal{H}(\bm{\rho}[a]\sqcup \bm{\rho}[1]^k,\bm{\rho}[b])& a>0>b, \\
       \mathcal{H}(\bm{\rho}[b]\sqcup \bm{\rho}[1]^k,\bm{\rho}[a])& b>0>a,\ \\
        \mathcal{H}(\bm{\rho}[1]^k,\bm{\rho}[a]\sqcup \bm{\rho}[b])&0>a,b.
    \end{cases} 
\end{equation*}
\end{cor}

Now that we have identified the sheaves $\mathcal{K}_k^{a,b}$ with the (pullback of) hypergeometric sheaves, we introduce a practical tool to simplify the hypergeometric sheaves $\mathcal{H}(\bm{\chi},\bm{\theta},\psi)$, valid over any finite field $F$, which is used in Section \ref{s: weakly gkr}.

\begin{lem}[\cite{Katzbook}, Thm. 8.4.10]
    Suppose $\bm{\chi}$ and $\bm{\theta}$ are distinct, then the semisimplification of $\mathcal{H}(\bm{\chi},\bm{\theta},\psi)$ is isomorphic to 
    \[\mathcal{H}(\bm{\chi}\smallsetminus\bm{\chi}\sqcap\bm{\theta},\bm{\theta}\smallsetminus\bm{\chi}\sqcap\bm{\theta},\psi) \oplus \bigoplus_{\rho\in \bm{\chi}\sqcap\bm{\theta}}G_{\rho}^{\otimes \operatorname{deg}}.\mathcal{L}_{\rho},\]
    where $G_{\rho}^{\otimes \operatorname{deg}}$ denote the geometrically constant sheaf associated to some algebraic number $G_{\rho}$ satisfying $\vert{G_{\rho}}\vert \leq \vert{k}\vert^{-1/2}$.
\end{lem}

The upshot of the Lemma is that the pure component of weight $0$ of $\mathcal{H}(\bm{\chi},\bm{\theta};\psi)$ is isomorphic to the pure component of weight $0$ of $\mathcal{H}(\bm{\chi}\smallsetminus\bm{\chi}\sqcap\bm{\theta},\bm{\theta}\smallsetminus\bm{\chi}\sqcap\bm{\theta};\psi)$ so that we can reduce our attention to those $\mathcal{H}(\bm{\chi},\bm{\theta})$ where $\bm{\chi}$ and $\bm{\theta}$ are disjoint. 

We delete the intersection of the two character sets in the hypergeometric sheaf from Corollary \ref{cor: K_k is hypergeometric} and associate the resulting hypergeometric sheaf to $\K_k^{a,b}$.

\section{Dyadic partition and P\'olya-Vinogradov method} \label{s: dyadic partition and poly vinogradov}
By Proposition \ref{prop: application AFE} we want to study the sums $M_{a,b,\Kl_k}$ and $M_{a,b,\K_k^{a,b}}$. In this section, we reduce the range of summation to dyadic intervals in the P\'olya-Vinogradov range $m\sim M$ and $n\sim N$ for $q^{1/2-\delta}\leq M,N\leq q^{1/2+\delta}$. We then show that it is sufficient to obtain type I bounds to conclude the general case in Theorem \ref{thm:introduction1}. Furthermore, we introduce a method to change the trace function that will be useful to treat some non-gallant cases in Section \ref{s: weakly gkr}.

Let $\delta \in (0,\frac{1}{100})$ denote a small number. Because of Remark \ref{rem: the FE and its consequences}, in this section for $k=1$ we exclude the cases $a$ or $b=-1$.
\begin{notation}\label{notation K}    
Let $u\in \FF_q^{\times}$. In this section, $K$ denotes either of the following trace functions: \begin{itemize}
     \item $K(u) = \Kl_k(u)$ or $K(u) = \K_k^{a,b}(u^{-1})$ if $k>1,$
     \item $K(u) = \K_1^{a,b}(u^{-1})$ if $-1\notin \{a,b\}$ 
     \item $K(u)=e_q(u)$ if $-1\notin \{a,b\}$ and $\{a,b\} \neq \{-2\}.$ 
 \end{itemize} 
 In each of the above cases $K$ is the trace function of a middle extension complex $\mathscr{K}$ of $\ell$-adic sheaves mixed of weight $0$, lisse on $\mathbb{G}_m$. We set  $K(0) = \operatorname{tr}(\operatorname{Frob}_0|V^{I_{(0)}}_{\mathscr{K}})$. We have $\Vert K\Vert_\infty \ll 1$, independently of $q$, but possibly depending on $a,b$ and $k$. 

\end{notation}

The goal of this section is to prove that there exists $\eta>0$ so that for $K$ as in Notation \ref{notation K} and for $K = e_q$, when $a=b=-2$, we have for any integer $\xi$ coprime to $q$ 
\begin{equation}\label{eq:goal of section 4}
    M_{a,b,K}(\xi)  
 = \frac{1}{q^{1/2}}\sum_{m,n}\frac{1}{(mn)^{1/2}}V(mn)K(\bar \xi m^{-a}n^{-b})\ll q^{-\eta},\end{equation}
 where $V\colon (0,\infty) \to \CC$ is any function satisfying \eqref{eq: decaying V 1} and \eqref{eq: decaying V 2}.
The case $K=e_q$ and $a=b=-2$ is treated at the end of the section.

 Applying a smooth dyadic partition of unity and by the rapid decay of $V$ we have
\begin{align*}
&M_{a,b,K}(\xi) \\&\ll \frac{(\log q)^2}{(qMN)^{1/2}} \sup_{\substack{M,N\\MN\leq q^{1+2\delta}}}\sum_{\substack{m,n\\(mn,q)=1}}K(\overline{\xi} m^{-a}n^{-b})U\left(\frac{m}{M}\right)U\left(\frac{n}{N}\right)V\left(\frac{mn}{q}\right) + O_{\delta}(q^{-100}),\end{align*}
where $U\in C_c^{\infty}(\RR_{>0})$ satisfies $\operatorname{supp}(U) \subset (1/2,2)$ and $\vert{x^jU^{(j)}(x)}\vert \ll_j 1$ for every $x\in \RR$. 
For $M,N\geq 1$ and $U$ as above, we denote \begin{equation}\label{eq: defi S_abK}
S_{a,b,K}(M,N)(\xi)=\sum_{\substack{m\\(m,q)=1}}\sum_{\substack{n\\(n,q)=1}}K(\bar \xi m^{-a}n^{-b})U\left(\frac{m}{M}\right)U\left(\frac{n}{N}\right)V\left(\frac{mn}{q}\right).\end{equation}
We do not display $U$ in the notation and throughout we might exchange $U$ with some $W$ that has the same properties $\operatorname {supp}(W) \subset (1/2,2)$ and $\vert x^jW^j(x) \vert \ll_j 1$ for any non-negative integer $j$.

\begin{lem}\label{lem: fourier}
    For a function $T$ on $\FF_q$ consider the naive Fourier transform 
    \[ \widehat{T}(u) = \frac{1}{q^{1/2}}\sum_{x\mods{q}}T(x) e_q(xu), \qquad u \in \FF_q.\]
    Let $K$ be as in Notation \ref{notation K}. 
    Let $K_{a,\xi}(u) = K(\xi u^{-a})$ for $\xi$ be an integer coprime to $q$. Then 
    \[ \Vert\widehat{K_{a,\xi}}\Vert_\infty  \ll 1, \]
    where the implied constant depends only on $a$ and $k$, but not on $q$ and $\xi$.
\end{lem}
\begin{proof}
    If $K(u)=\Kl_k(u)$ or $K(u) = e_q(u)$, then $K_{a,\xi}(u) = \Kl_k(\xi u^{-a})$ and $K_{a,\xi}(u) = e_q(u^{-a})$ (recall, in this case $-1 \neq a$) and in both cases the underlying sheaves are pure of weight $0$ and Fourier (\cite{Katz88}), so we use Laumon's theorem (see \cite{AppliedladicCohom}, Theorem 6.6, Proposition 6.7).
    
      Suppose now $K(u) = \K_k^{a,b}(u^{-1})$. If $a= -1 = b$ and $k=2$, then we see $\K_2^{-1,-1}(u) = q^{1/2}\delta_{u\equiv 1\mods{q}} + \K'(u)$, where $\Vert K' \Vert \ll q^{-1/2}$. The Fourier transform of $\K'$ is bounded trivially and the Fourier transform of $q^{1/2}\delta_{u\equiv 1\mods{q}}$ is just an additive character. If $a$ or $b$ is not $-1$ or $k\neq 2$, by Corollary \ref{cor: K_k is hypergeometric}, we can write $\K_k^{a,b}(u)  = \sum_{w \leq 0}(\K^{a,b}_k)_w(u)$, where $(\K^{a,b}_k)_w$ is a trace function, pure of weight $w$. For $w<0$ we bound the naive Fourier transform trivially.  The weight 0 component $(\K^{a,b}_k)_0$ is a hypergeometric function. For all allowed configurations, except $(a,-a,1)$ and $(-1,-1,3)$,  of $(a,b,k)$ one has by \cite{Katzbook}, Theorem 8.4.2 that the complex of sheaves underlying $(\K^{a,b}_k)_0(u)$ is Fourier and $(\K^{a,b}_k)_0(u) \neq e_q(u^{-1})$. Hence $K_{a,\xi}(u) = \K_k^{a,b}(\xi^{-1} u^{a})$ is again Fourier and we conclude as before. 
    For the triples $(a,-a,1)$ (in this case $a\neq 1$) and $(-1,-1,3)$, we have that  $$(\K_1^{a,-a})_0(u)=e_q(a^{2a}(-1)^au),\qquad \K_{3}^{-1,-1}(u)=e_q(u)$$ that give, respectively,
    $$ K_{a,\xi}(u) = e_q(\xi a^{-2a}(-1)^au^{-a}), \qquad K_{-1,\xi}(u) = K(\xi u)= e_q(\xi u^{-1}), $$
    whose underlying sheaves are Fourier.

\end{proof}

\begin{prop}\label{prop main under condition}Let $K$ be as in Notation \ref{notation K}. 
Assume that there exists $\eta >0$ (independent of $q$ and $k$) so that for all positive integers $M,N>0$ with $q^{\frac 12-2\delta}\leq M,N\leq q^{\frac 12+\delta}$ and $MN\geq q^{1-\delta}$ we have the bound
\begin{equation}\label{eq:desideratakloosterman}\frac{S_{a,b,K}(M,N)(\pm \xi)}{(qMN)^{1/2}} \ll q^{-\eta}. \end{equation}
 Then 
    $$M_{a,b,K}(\xi) \ll q^{-\min(\eta,\delta/2)}.$$
Here, the implicit constants are allowed to depend on $\eta$, $a,b$ and $k$.
\end{prop}

\begin{proof}

For $M,N$ not in the range $q^{\frac 12-2\delta}\leq M,N\leq q^{\frac 12+\delta}$ we want to show a bound of the shape 

\begin{equation*} 
\frac{S_{a,b,K}(M,N)(\pm \xi)}{(qMN)^{1/2}} \ll q^{-\eta}, \end{equation*}
    for some $\eta >0$. 
The trivial bound is $S_{a,b,K}(M,N)\ll MN$. Hence, if $MN< q^{1-\delta}$ for some $\delta >0$, then \eqref{eq:desideratakloosterman} is trivially valid with $\eta = \frac{\delta}{2}$. In particular, we assume that 
\begin{equation*}\label{eq:MNnonsmall}
q^{1-\delta} \leq MN \leq q^{1+2\delta}.
\end{equation*}
Suppose that $N>q^{\frac 12+\delta}$. Denote $\tilde{U}(x) = U(x)V(mx\frac{N}{q})$. Note that by \eqref{eq: decaying V 2}, we have $x^{j}\tilde{U}^{(j)}(x) \ll_j 1$. Then, by the Poisson summation formula, 
\[\sum_{n}K(\overline{\xi} m^{-a}n^{-b})\tilde{U}\left(\frac{n}{N}\right) = \frac{N}{q^{1/2}}\sum_{n}\hat{K}_{-a,m^{-b}\overline{\xi}}(n)\widehat{\tilde{U}}\left(n\frac{N}{q}\right). \]
From Lemma \ref{lem: fourier}, we have
$\Vert{\hat{K}_{-a,m^{-b}\xi}}\Vert_{\infty} \ll 1$, for a constant only depending on the complexity of the sheaf underlying $K$. Since $\widehat{\tilde{U}}(x) \ll_A \vert{x}\vert^{-A}$ ($x\neq 0$), we have   
\[\frac{N}{q^{1/2}}\sum_{n}\hat{K}_{-a,m^{-b}\pm\xi}(n)\widehat{\tilde{U}}\left(n\frac{N}{q}\right)\ll q^{1/2+o(1)}\]
and so by trivially estimating the $m$ contribution
\[\frac{S_{a,b,K}(M,N)(\xi)}{(qMN)^{1/2}}
\ll \frac{q^{o(1)}M^{1/2}}{N^{1/2}}  < q^{-\delta+o(1)}. \]
Suppose $N<q^{\frac 12- 2\delta}$, then $M \geq \frac{q^{1-\delta}}{N}\geq q^{1/2+\delta}$ and  
with the same proof we show that 
\[\frac{S_{a,b,K}(M,N)(\xi)}{(qMN)^{1/2}}
\ll \frac{q^{o(1)}N^{1/2}}{M^{1/2}}  < q^{-\delta+o(1)}.\]

\end{proof}

By Mellin inversion formula, we write $$V(y) = \frac{1}{2\pi i}\int_{\operatorname{real}(u)=\sigma}y^{-u}\mathcal{M}(V,u) \ \mathrm{d}u $$ for any $\sigma>0$. Here, $\mathcal{M}(V,s)$ is the Mellin transform of $V$ and it satisfies for any $\sigma >0$ and any $A>0$ \begin{equation}
    \mathcal{M}(V,\sigma+it) \ll_{\sigma,A} (1+\vert t\vert )^{-A}.  
\end{equation}
This can be deduced from \eqref{eq: decaying V 1}, \eqref{eq: decaying V 2} and integration by parts. 

 We write 
\begin{align}\label{eq : mellintrick}&S_{a,b,K}(M,N)(\xi) = \nonumber\\& \frac{1}{2\pi i}\int_{(\sigma)}\mathcal{M}(V,u)\left(\frac{q}{MN}\right)^u\sum_{m,n}K(\overline{\xi}m^{-a}n^{-b})\left(\frac{MN}{mn}\right)^u U\left(\frac{m}{M}\right)U\left(\frac{n}{N}\right) \ \mathrm{d}u. \end{align}
Since we assume that $MN>q^{1-\delta}$, this step introduces an error of size $q^{2\delta \sigma}< q^{\frac{   \sigma }{50}}$, which will be controlled by choosing $\sigma$ appropriately. 

Since the coefficients $(\frac{M}{m})^u U\left(\frac{m}{M}\right)$, $(\frac{N}{n})^u U\left(\frac{n}{N}\right)$ are smooth, we want to obtain type I bounds. Let $M,N>0$ and $(\alpha_m)_m$ be a sequence of complex numbers supported in $m\sim M$ with $\alpha_m \ll 1$ and $\mathcal{N}$ be a set of $N$ consecutive integers. We define 
\begin{equation}\label{eq: def S^I}
    S^{I}_{a,b,K}((\alpha_m),\mathcal{N})(\xi) = \sum_{m}\alpha_m\sum_{n\in \mathcal{N}}K(\bar \xi m^{-a}n^{-b}).
\end{equation}

In Section \ref{s: weakly gkr}, we show that a large number of $K$ which are included in Notation \ref{notation K} are gallant as defined in \cite{FKMSbilinear}. In this case, when $K$ is gallant, we can use Theorem 1.3 in \emph{loc. cit.}, which gives the following. Let $l\geq 3$ be an integer, then for any pair of real numbers $M,N$ that satisfy \[ q^{1/l}<\frac{N}{10}, \qquad M,N,\frac{N^2}{q^{1/l}} <q \]
for any sequence $(\alpha_m)$ supported in $m\sim M$ and any set $\mathcal{N}$ of $N$ consecutive integers and for any integer $\xi$ coprime to $q$ we have 
\begin{equation}\label{eq: type 1 gallant}
    S^I_{a,b,K}( (\alpha_m),\mathcal{N})(\xi) \ll q^{o(1)}\Vert (\alpha_m) \Vert_2 M^{1/2}N\left(\frac{q^{1+\frac{3}{2l}}}{MN^2} \right)^{\frac{1}{2l}}.
 \end{equation}
For $K = e_q$ we can directly use the same shifting trick used to prove \eqref{eq: type 1 gallant} to show, with $(\alpha_m),M,N,\mathcal{N}$ and $\xi$ as above but with $l\geq -b$, if $-b\geq 3$:
\begin{equation}\label{eq: type 1 eq}
    S^I_{a,b,e_q}((\alpha_m),\mathcal{N})(\xi) \ll q^{o(1)}\Vert  (\alpha_m) \Vert_2 M^{1/2}N\left(\frac{q^{1+\frac{1}{l}}}{MN^2} \right)^{\frac{1}{2l}}.
\end{equation}
  This is proved in Proposition \ref{prop:ab for e_q}.

In the next Proposition, we show that the type I bounds \eqref{eq: type 1 gallant} and \eqref{eq: type 1 eq} are sufficient to obtain the claimed power savings for the moments \[M_{a,b,K}(\xi) \ll q^{-\eta}.\] For $K$ gallant or $K=e_q$ (recall, $-1\notin \{a,b\}$ and if $a,b<0$, then at least one of the two is $\leq -3$) it is a simple application of integration by parts. For those $K$ which are not gallant, we can use the Poisson summation formula (see Lemma \ref{l:poissontrick}) to study the corresponding problem for a different trace function $K'$, which turns out to be gallant.

\begin{prop}\label{prop: bound with cut of fct}Let $K$ be as in Notation \ref{notation K} and suppose that $K$ is gallant or $K=e_q$ and, in the latter case, that $-b\geq 3$, if $a,b<0$. Then there exists $\eta >0$ (depending at most on $a$ and $b$) so that for every integer $\xi$ coprime to $q$ we have
    \[M_{a,b,K}(\xi) \ll q^{- \eta  + o(1)}. \]
\end{prop}

\begin{proof}  From Proposition \ref{prop main under condition} our task is to show that we can find $\eta >0$ (independent of $K$) so that for all positive integers $M,N$ so that $q^{1/2-2\delta}< M,N< q^{1/2+\delta}$ and $MN>q^{1-\delta}$ we have 
\[ \frac{S_{a,b,K}(M,N)(\xi)}{(qMN)^{1/2}} \ll q^{-\eta+o(1)}.\]
 Consider Equation \eqref{eq : mellintrick} and set $g(x) = x^{-u}U(x)$. By the properties of $U$ we have $xg'(x)\ll (\vert{u}\vert+1) \sup_{x\in [1/2,2]}\vert{x^{-u}}\vert$ and the supremum is bounded by $2^{\sigma} \ll 1$. 
    By summation by parts, we have 
    \begin{align*}
        \sum_{m,n}K(\overline{\xi}m^{a}n^{b})\left(\frac{MN}{mn}\right)^u U\left(\frac{m}{M}\right)U\left(\frac{n}{N}\right) =\int_{N/2}^{2N}S_{a,b,K}^{I}\left(g(\tfrac{m}{M}), \mathcal{N}_t\right)\frac{g'\left(t/N\right)}{N}\ \mathrm{d}t,
    \end{align*}
    where $\mathcal{N}_t$ is a set of $t-\tfrac{N}{2}$-consecutive integers.
    Note that the endpoints disappear, since $g\left(\frac{2N}{N}\right) = g(\frac{N}{2N}) = 0$. From \eqref{eq: type 1 gallant} and \eqref{eq: type 1 eq} we have that for $t>q^{1/2-3\delta}$ \[S_{a,b,K}^{I}\left(g(\tfrac{m}{M}), \mathcal{N}_t\right) \ll q^{o(1)} M t \left(\frac{q^{1+\tfrac{3}{8}}}{M(t-\tfrac{N}{2})^2}\right)^{\frac{1}{8}}. \]
    Let $C(q,M) = M\left( \frac{q^{1+\frac{3}{8}}}{M}\right)^\frac{1}{8}$.
    We split the integral into two parts, using the estimate $\Vert{K}\Vert_{\infty} \ll_{a,b,k} 1$, and obtain
    \begin{align*}&\sum_{m,n}K(\overline{\xi}m^{a}n^{b})\left(\frac{MN}{mn}\right)^u U\left(\frac{m}{M}\right)U\left(\frac{n}{N}\right)\\&\ll M\int_{N/2}^{N/2+q^{1/2-3\delta}}\vert{\frac{t}{N}g'(t/N)}\vert\ \mathrm{d}t + q^{o(1)}C(q,M)\int_{N/2+q^{1/2-3\delta}}^{2N}\vert{(t-\tfrac{N}{2})^{-1/4} \frac{t}{N}g'(t/N) }\vert\ \mathrm{d}t 
    \\ &\ll q^{o(1)}(1+\vert u\vert)(Mq^{\frac{1}{2}-3\delta}+MN\left(\frac{q^{1+\frac{3}{8}}}{MN^2}\right)^{1/8}\ll q^{o(1)}(qMN)^{\frac{1}{2}}(1+\vert u \vert)(q^{-\frac{3\delta}{2}}+q^{\delta+\frac{3\delta}{8}-\frac{1}{64}}). \end{align*}
    Let $\delta>0$ small enough so that $\eta(\delta) = \min(-\delta-\frac{3\delta}{8}+\frac{1}{64},\frac{3\delta}{2})>0$.
     Using the triangle inequality, we obtain 
\begin{align*}
    \frac{S_{a,b,K}(M,N)(\xi)}{(qMN)^{1/2}} & \ll_{a,b,k,\sigma} q^{2\delta \sigma-\eta(\delta)}\int_{\RR} \vert\mathcal{M}(V,\sigma+it)\vert(1+\vert t\vert) \ \mathrm{d}t \ll q^{-\eta(\delta) + \frac{1}{50}\sigma}, 
\end{align*}
by \eqref{eq : mellintrick}.
We conclude the proof by choosing $\sigma$ and $\delta$ small enough. 
\end{proof}

We now introduce a tool to handle some non-gallant cases, namely using Poisson summation to change the trace function without changing the length of the correlation sum. In most concerned cases, we can handle the new trace function, because it is gallant, see Lemma \ref{lem: special cases} for the application of this method. 
Let us start with the following observation. For $n,\xi\in \ZZ$ coprime to $q$, we have \begin{align}\label{algn;fourier transofrm}
            \widehat{\K_{k}^{a,b}(  \xi\bullet^b)}(n) &=\frac{1}{q^{1/2}}\sum_{x\in \FF_q} \K^{a,b}_{k}(  \xi x^b)e_q(xn)\nonumber\\
            &= \frac{1}{q^{\frac{k+2}{2}}}\sum_{\substack{x_1,x_2,y_1\ldots y_k \mods q\\x_1^ax_2^by_1\cdots y_k =   \xi}}e_q(x_1+y_1+\cdots y_k)\sum_{x \in \FF_q^{\times}}e_q(xn+xx_2) +\frac{1}{q^{\frac{1}{2}}}\K_k^{a,b}(0)\nonumber\\ 
            & = \frac{1}{q }\sum_{x_2 \in \FF_q^{\times}} \K^{a,1}_{k-1}(  \xi x_2^{-b})+\K^{a,1}_{k-1}( \xi (-n)^{-b})+\frac{1}{q^{\frac{1}{2}}}\K_k^{a,b}(0),
        \end{align}
where for $k=0$ we define \begin{equation}\label{eq: defi K_k^a,b for k=0}
\K_0^{a,b}(u) = \frac{1}{q^{1/2}}\sum_{x_1^ax_2^b=u}e_q(x_1+x_2).\end{equation} These functions are studied in \cite{FKMsecondmom} and correspond in \emph{loc. cit.} to $\widetilde{T}_{a,b}(u;q)$. They are associated to complexes of sheaves mixed of weight $0$ that are geometrically non-trivial if $a+b \neq 0$ (see Proposition 3.7 and Remark 3.8 in \emph{loc. cit.}). Hence, the first summand in \eqref{algn;fourier transofrm} is bounded by $q^{-1/2}$ if $a\neq -1$ or $k>1$ by the Riemann Hypothesis over finite fields (see \cite{AppliedladicCohom}, Corollary 4.7). 
Furthermore, for $n\equiv 0 \mods{q}$ one has 
\[ \widehat{\K_k^{a,b}(\xi \bullet^b)}(0 ) = \frac{1}{q^{1/2}}\sum_{x\in \FF_q}K_k^{a,b}(\xi x^b) \ll 1,\] 
again by the Riemann hypothesis over finite fields.

\begin{lem} \label{l:poissontrick} For $u\in \FF_q^{\times}$ let $K(u) = K_k^{a,b}(u^{-1})$.
Suppose that $a\neq -1$  or $k>1$ and denote $K_a(u) =  K_{k-1}^{a,1}(u^{-1})$.
    Then for any real numbers $M,N>0$ so that $q^{\frac{1}{2}-2\delta'}\leq M,N\leq q^{\frac{1}{2}+\delta'}$ and any integer $\xi$ coprime to $q$ we have 
    \[\frac{S_{a,b,K}(M,N)(\xi)}{(qMN)^{\frac{1}{2}}}\ll q^{\frac{3\delta'}{2}+o(1)}\sup_{q^{\frac{1}{2}-2\delta'}\leq M,N\leq q^{\frac{1}{2}+2\delta'}}\frac{S_{a,b,K_a}(M,N)(\xi)}{(qMN)^{1/2}} + O(q^{-\delta'}). \]
\end{lem}

\begin{proof} We show only the case where $a\neq -1$ or $k>1$.  
Let $m$ be so that $U(m/M)\neq 0$. By Poisson summation on $n$ we have 
\begin{align*}\sum_{n}U\left(\frac{n}{N}\right)V(nm)K_k^{a,b}(\xi m^an^b) = \frac{N}{q^{1/2}}\sum_{n\neq 0} \widehat{\tilde U}\left(\frac{nN}{q}\right)\bigg(K_{k-1}^{a,1}(\xi m^a(-n)^b)+O(q^{-1/2}) \\+\frac{N}{q^{1/2}}\widehat{\tilde U}(0)\widehat{K^{a,b}_{k}(m^a\bullet^b)}(0)\bigg), 
\end{align*}
where $\tilde{U}(x) = U(x) V\left(\frac{xmN}{q}\right)$. It satisfies $\operatorname{supp}(\tilde{U}) \subset (1/2,2)$ and $ x^j\tilde{U}^{(j)}(x)\ll_j 1$ for every $x$. The $0$-frequency contributes 
\[\frac{N}{q^{1/2}}\sum_{m}U\left(\frac{m}{M}\right)\widehat{\tilde U }(0) K_{k-1}^{a,1}(\xi m^a(-n^b)) \ll \frac{NM}{q^{1/2}}. \]
We split the non-zero frequencies using a smooth partition of unity and we see that
\[\frac{S_{a,b,K}(M,N)(\xi)}{(qMN)^{1/2}} \ll (\log q)\frac{N^{1/2}}{qM^{1/2}}\sup_{\tilde{N}\leq q^{1/2+2\delta'}}S_{a,b,K_{k-1}^{a,1}}(M,\tilde N)((-1)^b\xi) + \frac{N^{1/2}M^{1/2}}{q}.\]
The second term is bounded by $q^{-1/2+\delta'}$. It follows that 
\begin{align*}\frac{S_{a,b,K_k^{a,b}}(M,N)(\xi)}{(qMN)^{1/2}} \ll  (\log q )q^{\frac{\delta}{2}}  \sup_{\substack{q^{\frac{1}{2}-2\delta'}\leq M \leq q^{\frac{1}{2}+\delta'}\\0< \tilde N \leq q^{\frac{1}{2}+2\delta'}}}\frac{S_{a,b,K_{k-1}^{a,1}}(M,\tilde N)((-1)^b\xi)}{q^{3/4}M^{1/2}} \\+ O(q^{-1/2+\delta'}).  \end{align*}
For $\tilde{N}< q^{1/2-2\delta'}$ the contribution is trivially bounded by 
\[\frac{S_{a,b,K^{a,1}_{k-1}}(M,\tilde N)((-1)^b\xi)}{q^{3/4}M^{1/2}} \ll \frac{M^{1/2}\tilde N}{q^{3/4}} \leq q^{-\frac{3}{2}\delta'} \]
and so they are $O(q^{-{\delta'}+o(1)})$. 
This concludes the proof. 
\end{proof}

\begin{cor}\label{cor: conclusion chap 4}There exists $\eta >0$ so that for all $K$ as in Notation \ref{notation K} we have 
\[M_{a,b,K}(\xi) \ll q^{-\eta + o(1)}.\]
\end{cor}
\begin{proof} In what follows $0<\delta',\delta''<\delta$ are two small positive quantities. 
    If $K$ is neither gallant nor $e_q$, then $K(u) = \K_k^{a,b}(u^{-1})$ for some $(a,b,k)$ (see Section \ref{s: weakly gkr} for a complete list). If $\K_{k-1}^{a,1}$ or $\K_{k-1}^{1,b}$ is gallant, using Propositions \ref{prop main under condition} and \ref{prop: bound with cut of fct}, there exists $\eta >0$ (independent of $K$) so that  
    \[M_{a,b,K}(\xi) \ll q^{-\min(\eta-\frac{3\delta'}{2},\frac{\delta'}{2})} \]
    and we conclude by choosing $\delta'$ small enough.
    From Section \ref{s: weakly gkr}, there are only two tuples $(a,b,k)$ such that both $(a,b,k)$ and  both $(a,1,k-1)$, $(1,b,k-1)$ are not gallant, these are $(-1,-1,2)$ and $(-1,-2,2)$. 
    By Corollary \ref{cor: K_k is hypergeometric}, or by direct computation we see that $\K_2^{-1,-1}(u) = q^{\frac{1}{2}}\delta_{u\equiv 1\mods{q}}+O(q^{-1/2})$ and we see that for $K(u)=\K_2^{-1,-1}(u^{-1})$ 
    \[S_{-1,-1,K}(M,N)(\xi) \ll q^{1/2}\sum_{\substack{m,n\\ mn\equiv \xi\mods{q}}}U\left(\frac{m}{M}\right)U\left(\frac{n}{N}\right)V(mn) + \frac{MN}{q^{1/2}} \ll q^{\frac{1}{2}+o(1)}(1+\frac{MN}{q})\]
    and so 
    \[M_{-1,-1,\K_2^{-1,-1}}(\xi) \ll q^{-1/2+\delta+o(1)}.\]
    For the tuple $(-1,-2,2)$ we can apply Lemma \ref{l:poissontrick} and we see that it is sufficient to obtain bounds for $S_{-1,-2,K}(M,N)$, for the blocks $q^{1/2-2\delta''}<M,N<q^{1/2+2\delta''}$, where $K(u) = \K_1^{-1,1}(u^{-1})$. A direct computation shows that $\K_1^{-1,1}(u) = e_q(-u)+ O(q^{-1})$. In this case, we can use \eqref{eq: type 1 eq} and, similarly as before, we conclude by choosing $\delta''$ small enough.

\end{proof}

\subsection*{The subcase $K=e_q$ and $a=b=-2$}
The only case that is not covered by Proposition \ref{prop: bound with cut of fct} is the case $a=-2$, but what follows is true for any $a\neq -1$. In this case, we do not apply Proposition \ref{prop main under condition}, but rather group $mn$ into a single variable (called again $n$) and we get 
\begin{equation*}
    M_{a,a,e_q}(\xi) = \frac{1}{q^{1/2}}\sum_{\substack{n\geq 1\\(n,q)=1}}\frac{d(n)}{n^{1/2}}V\left(\frac{n}{q}\right)e_q(\xi n^{-a}).
\end{equation*}
As before, applying a smooth dyadic partition of unity and by the rapid decay of $V$, it suffices to show that there exists $\eta>0$ (independent of $q$) so that for any $N\leq q^{1+\delta}$ and any $U \in C_c^{\infty}(\left(\frac{1}{2},2\right))$ that satisfies $x^jU^{(j)}(x) \ll_j 1$, for all $j \in \ZZ_{\geq 1}$, we have
\begin{equation*}
    \frac{1}{(qN)^{1/2}}\sum_{n}d_2(n)U\left(\frac{n}{N}\right)e_q(\bar \xi{n^{-a}}) \ll q^{-\eta}.
\end{equation*}
If $N<q^{1-\delta}$ we apply trivial inequalities (with $d_2(n) \ll n^{o(1)}$) and get the bound 
\[\frac{1}{(qN)^{1/2}}\sum_nd_2(n)U\left(\frac{n}{N}\right)e_q(\bar\xi n^{-a}) \ll \frac{N^{1/2+o(1)}}{q^{1/2}} < q^{-\delta/2+o(1)}.\] Hence, we assume $N\geq q^{1-\delta}$. 

This type of problem is precisely studied in \cite{FKM_algebraic_over_primes} and, in particular, Theorem 1.15 in \emph{loc. cit.} says the following. For any isotypic \emph{non-exceptional}\footnote{As it is stated, the theorem is incorrent: one needs  also to assume that the sheaf $\mathcal{F}$ underlying the trace function is non-exceptional, as defined in \emph{loc. cit.}} trace function $K$ (of modulus $q$) and for any function $U\in C_c^{\infty}((\frac{1}{2},2))$, such that $x^jU^{(j)}\ll_j 1$ for all non-negative integers $j$, there exist $\eta >0$ and $A>0$ so that
\begin{equation*}
    \sum_{n\geq 1}d_{it}(n) U\left(\frac{n}{N}\right)K(n) \ll (1+\vert{t}\vert)^AN(1+\frac{q}{N})^{\frac{1}{2}}q^{-\eta'}.
\end{equation*}
Furthermore one can choose $\eta' < 1/8$ independently of $K$. Our case corresponds to the case $t=0$ and $K=e_q(\bar \xi \bullet^{-a})$ (which is a non-exceptional - recall $a \neq -1$ - isotypic trace function, since its associated sheaf is geometrically irreducible). Choosing $1/8>\eta' > \delta/2$ we get for any $q^{1-\delta}\leq N<q^{1+\delta}$  
\[ \frac{1}{(qN)^{1/2}}\sum_n d_2(n)U\left(\frac{n}{N}\right)e_q(\bar\xi n^{-a}) \ll q^{\delta/2-\eta'} \]
and so 
\[M_{a,a,e_q}(\xi) \leq q^{-\frac{\delta}{2}+o(1)}+q^{\frac\delta 2-\eta'+o(1)},\]
which proves \eqref{eq:goal of section 4}.

\section{The special case $k=1$}\label{S: Kl1}

In this section we show the bound on the bilinear sum stated in \eqref{eq: type 1 eq}. Recall that as in Notation \ref{notation K} we have $-1\notin\{a,b\}$ and $(a,b)\neq (-2,-2)$. We show the following (which implies \eqref{eq: type 1 eq}).

\begin{prop}\label{prop:ab for e_q}Suppose that $b\leq -3$ or $b\geq 1$. Let $l\geq 2$ be an integer so that if $-b\geq 3$, then $-b\geq l$. Let $M,N$ be real numbers satisfying \[q^{1/l}\leq \frac{N}{10},\qquad M,N,\frac{N^2}{q^{1/l}}<q.\]

Then for any $\xi \in \ZZ$ coprime to $q$, for any sequence $(\alpha_m)_m$ supported in $m\sim M$, any set $\mathcal{N}$ of $N$ consecutive integers we have
\begin{align*}
   \sum_{m\sim M}\alpha_m\sum_{n\in \mathcal N} e\left(\frac{\xi m^{-a} n^{-b}}q\right)
&\ll q^{o(1)}\|\alpha\|_2 M^{1/2}N\left(\frac{q^{1+\frac{1}{l}}}{MN^2}\right)^{1/2l}.
    \end{align*}
\end{prop}

\begin{lem}\label{lem: Pierces trick} Let $b$ and $l$ be two integers that satisfy $b<0$ or $b\geq l$. Let $\mathbf{v}\in \ZZ^{2l}$ be an integer vector. Let $q$ be a prime bigger than $b$ and $2l$. If the polynomial 
    $$P_\mathbf{v}(T)=\sum_{i=1}^l(T+v_i)^b-(T+v_{l+i})^b$$ vanishes identically on $\FF_q$, then we have that for all $1\leq i\leq l $ there exists $1\leq h\leq l$ such that $v_i\equiv v_{l+h}\mods q$. 
\end{lem}

\begin{proof}
We adapt semi-standard arguments that one can find, for example, in similar form, in \cite{ASENS_1998_4_31_1_93_0} or \cite{Pierce1}. Suppose first $b\geq l>0$. We see that $P_{\mathbf{v}}(r) = 0$ implies that \begin{equation*}
    p_k(v_1,\dots,v_l) \equiv  p_k(v_{l+1},\dots,v_{2l}) \mods{q}, \qquad 1\leq k \leq l,
\end{equation*} 
where $p_k(X_1,\dots,X_l) = \sum_{i=1}^l X_i^k$.
Let $e_0 =1$ and $e_1,\dots,e_l$ denote the elementary symmetric polynomials  
\[ e_i(X_1,\dots,X_l) = \sum_{1\leq j_1<\dots<j_i\leq l}X_{j_1}\cdots X_{j_i}.\]
By Newton's identities one has for any $1\leq k \leq l$ that
\[ ke_k = \sum_{i=1}^{k}(-1)^{i-1}e_{k-i}p_i.\]
In particular, we deduce that $e_i(v_1,\dots,v_l) \equiv  e_i(v_{l+1},\dots,v_{2l}) \mods{q}$ for each $1\leq i\leq l$ and so the desired statement.

Suppose now $b<0$. Write $\{v_1,\dots,v_l\} = \{v_1,\dots,v_{k}\}$ and let $m_j$, $1\leq j \leq k$ be the multiplicities. Similarly we write $\{v_{l+1},\dots,v_{2l}\} = \{v_{l+1},\dots,v_{l+k'}\}$ with multiplicities $n_j$, $1\leq j\leq k'$. Then 
\[P_\mathbf{v}(T) = \sum_{j=1}^k\frac{m_j}{(T-v_j)^{\vert b \vert}} - \sum_{j=1}^{k'}\frac{n_j}{(T-v_{j+l})^{\vert b\vert}}. \]
 If $P_\mathbf{v}(T) \equiv 0 \mods{q}$, then by clearing denominators, we have 
\[ \prod_{j=1}^{k'}(T-v_{l+j})^{\vert b\vert} \sum_{j=1}^km_j\prod_{\substack{i=1\\i\neq j}}^k(T-v_i)^{\vert b\vert} \equiv \prod_{j=1}^k(T-v_{j})^{\vert b\vert}\sum_{j=1}^{k'}n_j\prod_{\substack{i=1\\i\neq j}}^{k'}(T-v_{l+i})^{\vert b\vert} \mods{q}. \]
Now, $(T-v_1)^{\vert{b}\vert}$ divides the right-hand side and hence the left hand-side. Since $\FF_q[T]$ is a unique factorization domain and $(T-v_1)$ does not divide $\sum_{j=1}^km_j\prod_{\substack{i=1\\i\neq j}}^k(T-v_i)^{\vert b \vert}$, it must divide $\prod_{j=1}^{k'}(T-v_{l+j})^{\vert b\vert}$. Therefore, it is equal to $(T-v_{l+j})$ for some $j$. The same argument concludes the claim for $2\leq i \leq l$. 
\end{proof}

\begin{proof}[Proof of Proposition \ref{prop:ab for e_q}]

Without loss of generality we assume that $-b\geq 3$ or $b>1$. Let $l\geq 3$ and assume furthermore that if $-b$ is positive, we have $-b\geq l$. 
Let $V\geq 1$ such that $ V/10\leq N$ and $N^2/V\leq q.$ We first apply the $+uv$-shifting trick as stated in \cite{FKMSbilinear}, Proposition 4.1 to $\sum_{m}\alpha_m\sum_{n\in \mathcal N}e\left(\frac{ \xi m^{a}n^{b}}q\right)$ choosing $K(u)=e_q( \xi u) $. This gives

\begin{equation}\label{eq:boundmnsum}
    \sum_{m}\sum_{n\in \mathcal{N}}\alpha_m e\left(\frac{\xi m^{-a}n^{-b}}q\right) \ll q^{o(1)}\|\alpha\|_2 M^{1/2}N\left(\frac{1}{MN^2V^{2l-1}}\sum_{\mathbf{v}\in[V,2V]^{2l}}|\Sigma(\mathbf{v};e_q)|\right)^{\frac1{2l}},
\end{equation}
where 
\begin{align*}
    \Sigma(\mathbf{v};e_q)&=\sum_{(r,s)\in \FF_q\times\FF_q^\times}\prod_{i=1}^le_q\left(s(r+v_i)^{-b}\right)e_q\left(-s(r+v_{l+i})^{-b}\right) \\
     &=\sum_{r\in \FF_q}\sum_{s\in \FF_q^\times}e_q\left(s\sum_{i=1}^l(r+v_i)^{-b}-(r+v_{l+i})^{-b}\right).
\end{align*}

It is left to estimate $\sum_{\mathbf{v}\in[V,2V]^{2l}}|\Sigma(\mathbf{v};e_q)|$.
We denote by $P_\mathbf{v}(r)=\sum_{i=1}^l(r+v_i)^{-b}-(r+v_{l+i})^{-b}$ the rational function of degree $b-1$. Note that by orthogonality of characters we have 

$$\sum_{s\in \FF_q}e_q(sP_\mathbf{v}(r))=\begin{cases}
    q & \text{ when } P_\mathbf{v}(r)=0 \\
    0 & \text{ else.}
\end{cases}$$
Therefore, 
$$\Sigma(\mathbf{v};e_q)=q\cdot \#\{r : P_\mathbf{v}(r)=0\}-q.$$

If $P_\mathbf{v}$ vanishes identically over $\FF_q$ we have by Lemma \ref{lem: Pierces trick} that for all $i\leq 2l$ there exists $h\neq i$ such that $v_i\equiv  v_h\mods q$. Hence, there is a maximum of $O\left(V^l\left(\frac Vq+1\right)^l\right)$ such choices for the vector $\mathbf{v}$.
Since $V\leq 10N\leq 10q$ this contributes $O(V^lq^2)$ in the sum $\sum_{\mathbf{v}\in[V,2V]^{2l}}|\Sigma(\mathbf{v};e_q)|$.
If $P_\mathbf{v}$ does not vanish identically, then it has at most $O(\vert  b\vert )$ zeros. These cases contribute $O( \vert b\vert  qV^{2l})$ in the sum $\sum_{\mathbf{v}\in[V,2V]^{2l}}|\Sigma(\mathbf{v};e_q)|$. 

We now choose $V=q^{\frac{1}{l}}$, which balances the two contributions and satisfies $q^{\frac{1}{l}}<N$ for $l\geq 3$. With \eqref{eq:boundmnsum} we obtain the desired bound
\begin{align*}
\sum_{m}\sum_{n\in \mathcal{N}}\alpha_me\left(\frac{\xi m^{-a}n^{-b}}q\right) &\ll q^{o(1)}\|\alpha\|_2M^{1/2}N\left(\frac{q^{1+\frac{1}{l}}}{MN^2}\right)^{\frac1{2l}}. 
 \end{align*}

\end{proof}

\section{Bounds for trace functions of gallant sheaves}\label{s: weakly gkr}
We start this section by introduction the results of \cite{FKMSbilinear}. In particular we explain what it means to be gallant and present their type I bound on trace functions, valid if the trace function is associated to a gallant sheaf.
Then we use the classification regiment for monodromy groups due to Katz as presented in Section 9 of \cite{FKMSbilinear} to show that our hypergeometric sums are associated to gallant sheaves. At the end of this section we treat some special cases by standard arguments. Because of Remark \ref{rem: the FE and its consequences}, in this section we exclude the cases $k=1$ and $a$ or $b=-1$.

\subsection{Being gallant}
We recall the definition of gallant from \cite{FKMSbilinear}, Definition 1.2. 

\begin{defi}[Gallant groups]
    Let $E$ be an algebraically closed field of characteristic 0, let $r\geq 0$ be an integer and let $G\subset GL_{r,E}$ be a linear algebraic subgroup of $GL_r$ over $E$. The group $G$ is said to be gallant if the action of $G$ on $E^r$ is irreducible and moreover one of the following conditions is satisfied: 
    \begin{enumerate}
        \item the identity component $G^0$ of $G$ is a simple algebraic group (in particular, the integer $r$ is at least 2) 
        \item or the group $G$ is finite and contains a quasisimple normal subgroup $N$ acting irreducibly on $E^r$. 
    \end{enumerate}
\end{defi}

\begin{defi}[Gallant sheaves]\label{def weakly grk}
    Let $k$ be a finite field and let $\ell$ be a prime number invertible in $k$. Let $\mathcal{F}$ be an $\ell$-adic sheaf on $\mathbb{A}_k^1$ with generic rank $r\geq 0$. We say $\mathcal{F}$ is gallant if the geometric monodromy group of $\mathcal{F}$ is gallant as a subgroup of $GL_r(\bar \QQ_{\ell})$. 

    We say that $\mathcal{F}$ is gallant and light if $\mathcal{F}$ is gallant and moreover mixed of integral weights $\leq 0$ and its restriction to some open dense subset is pure of weight 0.
\end{defi}

In \cite{FKMSbilinear} Fouvry, Kowalsi, Michel and Sawin state type I and type II bounds. Since in our case we are able to deduce type II bounds from type I bounds (see Proposition \ref{prop: bound with cut of fct}) we only state their type I estimate. 

\begin{thm}[\cite{FKMSbilinear}, Thm. 1.3]\label{thm weakly gkr from fkms}
    Let $q$ be a prime and $K(\bullet)$ be the trace function of an $\ell$-adic sheaf $\mathcal{F}$ which is gallant and light in the sense of Definition \ref{def weakly grk} for some prime $\ell\neq q$. Let $\bm\alpha =(\alpha_m)_{m\leq M}$ be a sequence of complex numbers.
Let $a,b$ be non-zero integers, $l\geq 2$ be an integer and $M, N \geq 1$ integers. 

    \item Suppose that \begin{equation*}
   1\leq M\leq q,\ 10q^{1/l}\leq N\leq q^{1/2+1/(2l)}.
\end{equation*} 
Then we have
\begin{equation*}
\sum_{m\sim M}\sum_{n\sim N}\alpha_mK(m^an^b)\ll q^{\varepsilon}\|\bm \alpha\|_2M^{1/2}N\left(\frac{q^{1+3/(2l)}}{MN^2}\right)^{1/(2l)}.
\end{equation*}
Here the implicit constant depends on $\varepsilon,a,b,l$ and on the complexity $c(\mathcal{F})$ of $\mathcal{F}$. 

\end{thm}
To apply this result we need to check if our hypergeometric sums are related to gallant sheaves. 

Some examples are sheaves whose connected component of the monodromy group $SO_n, n=2,3 \text{ or }n\geq 5$, $Sp_n, n\geq 2$, $SL_n, n\geq 2$, $\mathfrak{S}_n, n\geq 5$, $\rho_7(G_2), \rho_8(Spin_7)$ or $\mathrm{Ad}(SL_3)$, see Section 9 in \cite{FKMSbilinear}.    
The probably most important exception from gallant groups is $SO_4$. It is not simple, hence not gallant. For finite groups an important counterexample is the symmetric group of order $n$ for $n$ small. It is gallant only as soon as $n\geq 5$ (see \cite{FKMSbilinear}, Section 9).

\subsection{Katz classification }
In the work of Katz, \cite{Katzbook}, he classifies the monodromy groups of hypergeometric sums based on the character set attached to it.
By Section \ref{S: Hasse- Davenport} we have reduced the analysis of $\K_k^{a,b}$ to studying hypergeometric sums. For $\Kl_k$ the sum is already in the correct shape. We recall the definition of hypergeometric sums. For two disjoint sets of characters $\bm\chi$ and $\bm \theta$ we have

$$\Hyp(u,\bm \chi,\bm \theta,\psi)=\frac{1}{|F|^{\frac{r+t-1}2}}\sum_{\substack{x_1,\ldots,x_r,y_1,\ldots,y_t\\\frac{x_1\cdots x_r}{y_1\cdots y_t}=u}}\psi(x_1+\ldots+x_r-y_1-\ldots-y_t)\prod_{i=1}^r\chi_i(x_i)\prod_{i=1}^t\bar\theta_i(y_i).$$
In the following we denote $r=|\bm\chi|$, $t=|\bm\theta|$ and $n=\max (r,t)$.
Note that by Lemma \ref{l: hasse davenport} we know that $\bm \chi$ is the set of characters $\bm \rho[1]^k$ in the case of $\Kl_k$ united with $\bm \rho[a]$ and/or $\bm \rho[b]$ if $a$ and/or $b$ are positive for $\K_k^{a,b}$. Similarly the set $\bm \theta$ consists of $\bm \rho[a]$ and/or $\bm \rho[b]$ for $a$ and/or $b$ negative. By \cite{FKMSbilinear}, Remark 9.2 without loss of generality we assume $\bm \chi$ to be non-empty.

We distinguish two cases. The first case is if $\bm \theta=\emptyset$ and $r\geq 2$. In this case we are reduced to studying the generalized Kloosterman sums
$$\Kl(u,\bm \chi,\psi)=\frac{1}{|F|^{\frac{n-1}2}}\sum_{\substack{x_1,\ldots,x_n\\x_1\cdots x_n=u}}\psi(x_1+\ldots+x_n)\prod_{i=1}^n\chi_i(x_i).$$
 As usual, we assume $k>0$. Then the case of only one set of characters occurs for $a>0, b>0$, for $a<-1,b<-1,k=1$, for $a<0,b<0,k=2$, for $a>0,b<0, b|a$ where we require $b<-1$ if $k=1$ and of course for $\Kl_k$ with $k\geq 2$.
The case in which $\bm \chi$ and $\bm\theta$ are two non-empty disjoint sets of characters occurs for $a>0,b<0, b\nmid a$ and $a<0,b<0,k>2$.

\begin{defi}
The multiset $\bm \chi$ is Kummer induced if there exists a $d>1$ dividing $r=|\bm\chi|$ and a character $\eta$ of order $d$ such that $\bm \chi =\eta \bm\chi$ as multisets. 
\end{defi}

By \cite{Katzbook}, Remark 8.9.3 a pair of character sets is Kummer induced if both sets are individually Kummer induced of degree $d$. 
\begin{defi}
A pair of character sets $(\bm \chi, \bm \theta)$ is called
\begin{enumerate}
    \item Kummer induced if there exists $d>1$ of such that for every character $\eta$ of exact order $d$, both $\bm \chi$ and $\bm \theta$ as sets with multiplicity are stable under the operation $\xi\mapsto \eta \xi$.
    \item Belyi induced if $r=t=n$ and there exist $c,d\in \NN$ and characters $\alpha,\beta$ with $\beta\neq 1$ such that $c+d=n$ and 
 \begin{align*}
    &\bm \chi =\{\text{all $c$'s roots of }\alpha\}\cup \{\text{all $d$'s roots of }\beta\} \\
    &\bm \theta = \{\text{all $n$'s roots of }\alpha\beta\}. 
     \end{align*}
\end{enumerate}
\end{defi}

Let $\Lambda_{\bm\chi}=\prod_{\chi\in \bm\chi}\chi$ and $\Lambda_{\bm\theta}=\prod_{\theta\in \bm\theta}\theta.$ We also recall the following definitions motivated by the work of Katz.

\begin{defi}
    \begin{enumerate}
        \item  A multiset $\bm{\chi}$ is twist-selfdual with dualizing character $\xi$ if the multiset of characters $\bm{\chi}$ is stable under $\chi \mapsto \xi\chi^{-1}$ i.e.  $$\bm{\chi}=\xi\bm\chi^{-1}.$$ 
\item  A twist-selfdual multiset $\bm{\chi}$ is symmetric if either $n$ is odd or $n$ is even and $\Lambda_{\bm \chi}\neq \xi^{n/2}$. If $\bm \chi$ is not symmetric, it is called alternating.
    \end{enumerate}
\end{defi}

With this we introduce Katz's classification theorems. The following two theorems and proofs are summaries of Katz's work, drawn up by the authors of \cite{FKMSbilinear} known to the authors of this work by private communication. We denote by $G^{0,der}$ the derived subgroup of the connected component. 

\begin{thm}[Katz]\label{thm: Katz one set characters}
Assume that $n\geq 2$, that $q>2n+1$ and that $\bm{\chi}$ is not Kummer induced. 
We have
\begin{enumerate}
    \item  If $n$ is odd or if $n$ is even and $\bm{\chi}$ is not twist-selfdual then $G^0 = G^{0,der} = SL_n$.
\item  If $n$ is even and $\bm{\chi}$ is twist-selfdual and symmetric then $G^0 = G^{0,der} = SO_n$.
\item  If $n$ is even and $\bm{\chi}$ is twist-selfdual and alternating then $G^0 = G^{0,der} = Sp_n$.
\end{enumerate}
\end{thm}
\begin{proof}
This is \cite{Katzbook}, Thm. 8.11.3 and Prop. 8.11.5 (with d = N = n in the notations of loc.
cit) along with the duality recognition criteria \cite{Katzbook}, Thm. 8.8.1 and Thm. 8.8.2.
\end{proof}
Let $\chi_{1/2}$ be the unique character of order 2. 
For the case of two character sets we have the following classification.

\begin{thm}[Katz]\label{thm: Katz two set characters}
 Assume that $\bm\chi$ and $\bm\theta$ are disjoint, that $(\bm\chi,\bm\theta)$ is not Kummer induced, and
that, in case $r = t$, neither $(\bm\chi,\bm\theta)$ nor $(\bm\theta,\bm\chi)$ is Belyi induced. The following holds:
\begin{enumerate}
    \item  The group $G$ is reductive and $G^0 = G^{0,der}$.
\item  If $r - t \equiv 1 \mods{2}, \ G^0 = SL_n$.
\item  If $r - t \equiv 0 \mods 2$ with $r \neq t$ then $G^0$ is either $SL_n, SO_n$ or $Sp_n$ (if $n$ is even) 
unless $|r - t| = 6$ and $n = 7, 8, 9$ in which case $G^{0,der}$ could also be
\begin{itemize}
    \item  $n = 7: \ \rho_7(G_2)$
    \item $n = 8: \ \rho_8(Spin_7)$ or $\mathrm{Ad}(SL_3)$ or $\mathrm{st}^{\otimes 3}(SL_2 \times SL_2 \times SL_2)$.
\item  $n=9: \ \mathrm{st}^{\otimes 2}(SL_3\times SL_3).$
\end{itemize}
\item  If $r - t \equiv 0 \mods 2$ with $r\neq t$ and both $\bm\chi$ and $\bm\theta$ are twist-selfdual with the same
dualizing character then $G^0 \subset SO_n$ unless $n \equiv 0 \mods 2$ and $\Lambda_{\bm\chi}/\Lambda_{\bm\theta}=1$ in which
case $G^0 \subset Sp_n$ .
\item  If $r = t$, $G^0$ is either $\{1\}, SL_n, SO_n$ or $Sp_n$ and
\begin{itemize}
    \item  if $\Lambda_{\bm\chi}/\Lambda_{\bm\theta}=1$, $G^0 = SL_n$ or $Sp_n$ (if $n$ is even),
\item if $\Lambda_{\bm\chi}/\Lambda_{\bm\theta}=\chi_{1/2}$, $G^0 = \{1\}, SO_n$ or $Sp_n$.
\end{itemize}
\end{enumerate}
\end{thm}
\begin{proof}
(1) and (2) are contained in cases (1) and (2) of \cite{Katzbook}, Thm. 8.11.3; (3) is \cite{Katzbook}, Thm.
8.11.3 (3); (4) is a consequence of the duality recognition criteria \cite{Katzbook}, Thm. 8.8.1 and
Thm. 8.8.2: under our assumption if $\xi$ denote the common dualizing character, the twist
$\mathscr L_\xi\otimes \mathscr H(\bm \chi ,\bm \theta;\psi)=\mathscr H (\xi\bm \chi,\xi\bm\theta;\psi)$ is twist-selfdual with symmetric or alternating twist-selfduality. (5) is contained \cite{Katzbook}, Thm. 8.11.2 and Cor. 8.11.2.1.\end{proof}

By Section 9 in \cite{FKMSbilinear} we know which of these cases are gallant. 
First we need to check if our cases are neither Kummer nor Belyi-induced.

\begin{lem} In the case of $\bm \theta =\emptyset$ the corresponding character set of the hypergeometric sum associated to $\K_k^{a,b}$ is not Kummer induced. This concerns the following cases: 
 \begin{enumerate}
      \item for $a>0, b>0$, 
      \item for $a-1,b<-1,k=1$,
      \item for $a<0,b<0,k=2$
      \item for $a>0,b<0, b|a$ with $b<-1$ if $k=1$ 
  \end{enumerate} 
The character set associated to $\Kl_k$, $k>1$ is not Kummer induced either. 
\end{lem}
\begin{proof}
\begin{enumerate}
    \item 
   The corresponding character set is $$\bm\chi=\bm \rho[a]\sqcup\bm \rho[b]\sqcup\bm \rho[1]^k.$$ It contains the trivial character at least three times. No other character is contained three times in $\bm \chi$, hence it can not be Kummer induced. 

  \item   The corresponding character set is
    $$\bm\chi=\bm \rho[a]\sqcup\{\bm \rho[b]\smallsetminus\{1\}\}.$$
For $b<-1$ we know that $d$ has to divide $(a,b)$ for the character set to be possibly Kummer induced. But we also know that $d$ divides $|\bm \chi|=a+b-1$. Hence $d=1$ and $\bm\chi$ is not Kummer induced. 

  \item   The corresponding character set is
    $$\bm\chi=\{\bm \rho[a]\smallsetminus\{1\}\}\sqcup\{\bm \rho[b]\smallsetminus\{1\}\}.$$
 In this set the characters that occur twice are exactly $\bm {\chi'}= \bm \rho[(a,b)]\smallsetminus\{1\}$. Hence this set itself would need to satisfy $\bm {\chi'}=\eta \bm {\chi'}$. But this is only possible with $\eta$ being the trivial character. 

  \item   The corresponding character set is 
    $$\bm\chi=\{\bm \rho[a]\smallsetminus\bm \rho[b]\}\sqcup\bm \rho[1]^{k}.$$
   For $k>1$ we use the same argument as in the first case. If $k=1$ and $b<-1$, the only character that keeps $\bm \rho[a]\smallsetminus\bm \rho[b]\sqcup \{1\}$ stable is the trivial character.
\end{enumerate}
   For $\Kl_k$ with $k>1$ the associated character set is $\bm \rho[1]^k$, which is not Kummer induced. 
\end{proof}

\begin{rem}
    Note that the character set associated to $(a,-1,1)$ (which we excluded due to Remark \ref{rem: the FE and its consequences}) is Kummer induced. 
\end{rem}

\begin{lem}\label{l: Not Kummer or belyi}
 For $a>0,b<0, b\nmid a$ and $a<0,b<0,k>2$ the corresponding pairs of character sets of the hypergeometric sums associated to $\K_k^{a,b}$ are neither Kummer nor Belyi induced.
\end{lem}

\begin{proof}
    In the first case our sets of characters are 
    $$\bm\chi=\{\bm \rho[a]\smallsetminus\bm \rho[(a,b)]\}\sqcup\bm \rho[1]^{k} \text{ and }\bm \theta=\bm \rho[b]\smallsetminus\bm \rho[(a,b)].$$
    The set $\bm \chi$ and hence the pair $(\bm \chi,\bm \theta)$ is not Kummer induced. 
    Suppose $(\bm \chi,\bm \theta)$ is Belyi induced. Since $\beta\neq 1$ the trivial character can only appear once in $\bm \chi$ as a root of $\alpha=1$. Hence necessarily $k=1$ and $b=-a-1$. Our character sets in this case are $(\bm \rho[a],\bm \rho[b]\smallsetminus\{1\})$. By definition every element $\theta\in \bm \theta =\bm \rho[b]\smallsetminus\{1\}$ is a $b-1$-th root of unity. Hence 
    $$\beta= \theta^{b-1}=\theta^{-1} \text{ for all }\theta \in \bm \rho[b]\smallsetminus \{1\}.$$
This forces $b=-2$. Hence our character sets are $(\{1\},\{\chi_{1/2}\})$ but this is not Belyi induced. The tuple $(\bm\theta,\bm\chi)$ is not Belyi induced because for $\beta\neq 1$ the $n-th$ roots of $\alpha\beta$ do not include the trivial character. 
    
    In the second case our sets of characters are 
    $$\bm\chi=\bm \rho[1]^{k-2} \text{ and }\bm \theta=\{\bm \rho[b]\smallsetminus\{1\}\}\sqcup\{\bm \rho[a]\smallsetminus\{1\}\}.$$ For $k>2$ it is clear that $\bm \chi$ is not Kummer induced. Suppose $(\bm \chi,\bm \theta)$ is Belyi induced. Similarly to the first case it is necessary that $k=3$ and $a+b=-3$. Our character sets are then $(\{1\},\{\chi_{1/2}\})$ but this is not Belyi induced.  
    The tuple $(\bm\theta,\bm\chi)$ is not Belyi induced because since $\beta\neq 1$ the $n-th$ roots of $\alpha\beta$ do not include the trivial character.
\end{proof}

\begin{lem} \label{lem: gallant unles...}
    The sheaf $\mathcal{F}$ associated to the trace function $\Hyp(u,\bm\chi,\bm \theta,\psi)$ associated to the triple $(a,b,k)$ is gallant, unless \begin{itemize}
        \item $(|\bm \chi|,|\bm \theta |)=(r,t)=(8,2)$ or $(9,3)$ (potentially $ \mathrm{st}^{\otimes3}(SL_2\times SL_2\times SL_2)$ or $ \mathrm{st}^{\otimes2}(SL_3\times SL_3)$)
        \item $n=\max(r,t)=4$ (potentially $SO_4$)
        \item $n=1$ (rank 1)
        \item $(a,b,k)=(1,-2,1),(2,-3,1)$ or $(3,-4,1)$ (solvable groups $\mathfrak{S}_2,\mathfrak{S}_3, \mathfrak{S_4}$).
    \end{itemize}
\end{lem}
\begin{proof}
This follows from \cite{FKMSbilinear}, Theorem 9.1, 9.2 and 9.3 and their Section 9.7. 
\end{proof}

\subsection{Special cases}

Lemma \ref{lem: gallant unles...} leaves us only having to check a few special cases. We start with the cases $r=8$ or $r=9$ and then introduce a trick to handle the cases $r=4$, finite groups and some rank 1 cases.  

In the following we continue to denote by $\chi_{1/2}$ the unique character of order 2. We furthermore denote by $\chi_{1/3}$ and $\chi_{2/3}$ the two characters of order 3, by $\chi_{1/4},\chi_{3/4}$ the two characters of order 4 and analogously for higher orders.

\begin{lem}
    For $(r,t)=(8,2)$ or $(9,3)$ the monodromy groups $\mathrm{st}^{\otimes3}(SL_2\times SL_2\times SL_2)$ and $ \mathrm{st}^{\otimes2}(SL_3\times SL_3)$ do not occur.
\end{lem}

\begin{proof}
This is an exhaustive list of possible cases for $(r,t)=(8,2)$ or $(r,t)=(9,3)$.
\begin{center}
    \begin{minipage}{0.75\textwidth}
 \vspace{1pt}
\begin{tabular}{cc}
cases $(r,t)=(8,2)$ & cases $(r,t)=(9,3)$\vspace{5pt} \\
 \begin{tabular}{cc}
\begin{tabular}{c|c|c}    
    $a$&$b$&$k$ \\
    \hline
        1 & -3 & 8 \\
        2 & -3 & 7 \\
        4 & -3 & 5 \\
        5 & -3 & 4 \\
        7 & -3 & 2 \\
        8 & -3 & 1 \\
        2 & -4 & 8 \\
        6 & -4 & 4 \\
        1 & -9 & 2 \\
        2 & -9 & 1 \\
        2 & -10 & 2 \\
        4 & -12 & 2 \\
        8 & -16 & 2 \\
    \end{tabular} \vspace{10pt}&
     \begin{tabular}{c|c|c}    
    $a$&$b$&$k$ \\
    \hline
        -1 & -9 & 4 \\
        -2 & -8 & 4 \\
        -3 & -7 & 4 \\
        -4 & -6 & 4 \\
        -5 & -5 & 4 \\
        -1 & -3 & 10 \\
        -2 & -2 &10\vspace{84pt} \\
  \end{tabular} \\
     \end{tabular} &
 \begin{tabular}{cc}
\begin{tabular}{c|c|c}    
    $a$&$b$&$k$ \\
    \hline
         1 & -4 & 9 \\ 
        3 & -4 & 7 \\
        5 & -4 & 5 \\
        7 & -4 & 3 \\
        9 & -4 & 1 \\
        3 & -6 & 9 \\
        9 & -6 & 3 \\
        1 & -10 & 3 \\
        3 & -10 & 1 \\
        3 & -12 & 3 \\
        9 & -18 & 3\vspace{37pt}\\
    \end{tabular}&
     \begin{tabular}{c|c|c}    
    $a$&$b$&$k$ \\
    \hline
       -1 & -10 & 5 \\
        -2 & -9 & 5 \\
        -3 & -8 & 5 \\
        -4 & -7 & 5 \\
        -5 & -6 & 5 \\
        -1 & -4 & 11 \\
        -2 & -3 & 11\vspace{93pt} \\
  \end{tabular} \\
     \end{tabular} \\
\end{tabular}    
\end{minipage}
\end{center}

    For $n=8$ we check that the the group $\mathrm{st}^{\otimes 3}(SL_2\times SL_2\times SL_2)$ does not occur. Using \cite{Katzbook}, Theorem 10.8.1 we obtain that for this group the two sets of characters have to be a twist of sets of characters of the following shape: 
Let $\xi$ be a character not of exact order four.
Assume wlog that $|\bm\chi|=8$. Then we need that 
$$\bm{\chi}=\{\xi,\bar\xi\}\cup \{\text{all cube roots of }\xi,\bar\xi\}\text{ and } \bm\theta=\{\chi_{1/4},\chi_{3/4}\} \text { or }\bm\theta=\{\chi_{1/3},\chi_{2/3}\}.$$
The condition $\bm\theta=\{\chi_{1/4},\chi_{3/4}\}$ on $\bm \theta$ gives $b=-4$. 
With twists the only other possibility is $a=2,b=-9,k=1$ with a twist by $\chi_{1/4}$. In this case the set $\chi_{1/4}\bm\chi$ does not have the required shape. 
For $b=-4$ the sets do not have the required shape either. (Note that for $a=6, k=4$ the character $\chi_{1/2}$ is not in $\bm{\chi}$.)
The condition $\bm\theta=\{\chi_{1/3},\chi_{2/3}\}$ gives $b=-3$. There is no possibility for a twist. In all cases the set $\bm\chi$ does not have the required shape since it contains the trivial character at least once. This fixes $\xi$ to be 1, giving the set $\{1,1,1,1,\chi_{1/3},\chi_{1/3},\chi_{2/3},\chi_{2/3}\}$  which is not possible.

For $n=9$ the we check for the group $\mathrm{st}^{\otimes2}(SL_3\times SL_3)$. From \cite{Katzbook}, Theorem 10.9.1 we obtain that for this group the two sets of characters have to be a twist of sets of characters of the following shape: 
Let $\xi_1,\xi_2,\xi_3$ be characters with orders not dividing three such that $\xi_1\xi_2\xi_3=1$.
Let $\chi_{1/3}$ and $\chi_{2/3}$ be the two characters of exact order three. Assume wlog that $|\bm\chi|=9$. Then we need that 
$$\bm{\chi}=\{\xi_1,\xi_2,\xi_3\}\cup \{\text{both square roots of }\bar \xi_1,\bar \xi_2\text{ and }\bar \xi_3\}\text{ and } \bm\theta=\{1,\chi_{1/3},\chi_{2/3}\}.$$
The condition on $\bm \theta$ excludes all cases except $a=3,k=1,b=-10$. But $\bm \rho[10]\smallsetminus\{1\}$ includes $\chi_{1/2}=\bar\chi_{1/2}$ which has square roots not included in $\bm \rho[10]$ and is a square root only of $1$ (which has an order dividing 3). Hence it is not possible to write the characters in the required shape. Twisting $\bm{\theta}$ by $\chi_{1/6}$ allows the options $a=3,k=9,b=-6$ and $a=9,k=3,b=-6$. In the first case $\bm{\chi}$ consists only of trivial characters and in the second case there are 3 trivial and 6 non-trivial characters. Neither of these are possible. 

\end{proof}

For the remaining cases we will apply Lemma \ref{l:poissontrick}. This changes the trace function from $\K_k^{a,b}$ to $\K_{k-1}^{a,1}$ or $\K_{k-1}^{1,b}$ and in most of our cases the new trace function is then associated to a gallant sheaf. The only tuples for which this does not hold are the two cases $(-1,-1,2)$ and $(-1,-2,2)$. These are already treated explicitly in Corollary \ref{cor: conclusion chap 4}.

\begin{lem}\label{lem: special cases}
Assume that $(a,b,k)$ is such that 
\begin{align*}
    n&=\max (r,t)=  4\text{  or}\\
    (a,b,k)&\in \{(1,-2,1), (2,-3,1),(-4,3,1),(-1,-1,3),(-1,-2,3),(a>1,-a,1)\}.
\end{align*}
Then for $a\neq -1$ or $k>1$ the function $u\mapsto \K_{k-1}^{a,1}(u^{-1})$ is the trace function of a gallant sheaf. 

\end{lem}

\begin{proof}
The reduction is done using Lemma \ref{l:poissontrick}. We then identify the new connected component of the monodromy group using Theorem \ref{thm: Katz one set characters}. This is a list of all cases.   

\centering{
 \vspace{15pt}
    \begin{tabular}{c|c|c| c c c|c|c|c}    
    $a$&$b$&$k$&$G^0$ or problem&reduces to &$a$&$b$&$k$ &new $G^0$ \\
    \hline
   2&1&1&$SO_4$ & & 2&1&0& $SL_3$\\
   -3&-2&1&$SO_4$ &&-3&1&0 &$Sp_2$\\ 
   2&-1&3&$SO_4$ && 2&1&2 & $SL_5$ \\
   2&-3&3&$SO_4$&&2&1&2&$SL_5$ \\
   4&-3&1&$SO_4$&&4&1&0&$SL_5$ \\
   2&-5&1&$SO_4$&&2&1&0 &$SL_3$ \\
   2&-5&3&$SO_4$ or $Sp_4$&&2&1&2&$SL_5$\\
    3&-6&4&$SO_4$&&3&1&3 &$SL_7$ \\
    6&-3&1&$SO_4$&&6&1&0 &$SL_7$ \\
    1&-2&1& $\mathfrak{S}_2$&& 1&1&0&$Sp_2$\\
     2&-3&1& $\mathfrak{S}_3$&&2&1&0&$SL_3$\\
     -4&3&1& $\mathfrak{S}_4$&&-4&1&0&$SL_3$\\
    -1&-1&3&rank 1&&-1 &1&2&$Sp_2$ \\
    -1&-2&3&rank 1&&-1 &1&2&$Sp_2$ \\
     $a>1$&$-a$&1&rank 1&& $a$&1&0&$\begin{cases}
        SL_{a+1} & \text{ if $a$ even} \\
        Sp_{a+1}  & \text{ if $a$ odd}
    \end{cases}$\\

    \end{tabular}}
    \vspace{15pt}
    
Choosing $\delta$ in Lemma \ref{l:poissontrick} small enough it now suffices to bound the correlation sum of the new trace function. These are all gallant. \qedhere
\end{proof}
\begin{rem}
    The case $(-1,-2,3)$ and $(-1,-1,3)$ can alternatively be treated by applying the functional equation or by following the recipe in \cite{FKMSbilinear}, Section 9.10. 
    The cases that have $G^0=SO_4$ are actually almost all (except $(2,-5,1)$ and $(3,-6,4)$) so called oxozonic cases (i.e. $G=O_4$). They can be treated by the standard procedure even though they are not gallant, see \cite{FKMSbilinear}, Section 8 and Theorem 9.3.
\end{rem}

\section{Proof of the asymptotic formulae}\label{s: Adaptation second moment}
In this section we adapt the asymptotic formulae for the second toroidal moment for $k=0$ of \cite{FKMsecondmom} to complete the proof of the two special cases in Theorem \ref{thm:introduction1}. We furthermore apply the bound on trace functions of gallant sheaves introduced in Theorem \ref{thm weakly gkr from fkms} to prove the general case of Theorem \ref{thm:introduction1}. We start with the adaptation of \cite{FKMsecondmom}. They show the following asymptotics. 

\begin{thm}\label{thm second moment k=0}[\cite{FKMsecondmom}, Thm. 1.1]
Let $a,b\in \ZZ\smallsetminus\{0\}$ and $q$ a prime number. There exists and absolute and effective constant $c>0$ and a real number $C$ such that, defining 
$$ \delta=\frac{c}{|a|+|b|},$$ we have, the following asymptotic formulae: 
\begin{enumerate}
    \item If $a+b=0$ then 
    $$\frac{1}{q-1}\sum_{\chi \mods q}L(\chi^a,\tfrac 12)L(\chi^{-a},\tfrac 12)=\log q+2C+O(q^{-\delta}).$$
    \item If $a+b\neq0$ then 
       $$\frac{1}{q-1}\sum_{\chi \mods q}L(\chi^a,\tfrac 12)L(\chi^{b},\tfrac 12)=\begin{cases}
           \zeta\left(\frac{|a|+|b|}{2(a,b)}\right)+O(q^{-\delta}) &\text{ if }ab<0 \\
           1+O(q^{-\delta}) &\text{ if }ab>0.
       \end{cases}$$
\end{enumerate}
\end{thm}
Since we will adapt the proof of this result, we give a quick outline of their strategy. 
After applying the approximate functional equation and splitting into even and odd characters in Section 5, they need to handle two sums, 
\begin{align*}
           N_{a,b}(X)&=\frac 12 \sum_{\substack{m,n\geq 1\\ (m^an^b)^2\equiv 1\mods q}}\frac{1}{\sqrt{mn}}V\left(\frac{mn}{X}\right) \\
    \text{and  }   P_{a,b}(Y)&=\frac{1}{2\sqrt q}\sum_{m,n\geq 1}\frac{1}{\sqrt{mn}}\widetilde{T}_{2a,2b}(m^{2a}n^{2b};q)V\left(\frac{mn}{Y}\right), \\
    \text{where  } 
\widetilde{T}_{a,b}(u;q)&=\frac{1}{\sqrt q}\sum_{\substack{x,y\in \FF_q^\times \\x^ay^b\equiv u \mods q}}e\left(\frac{x+y}{q}\right).
\end{align*}
The main term is extracted from $N_{a,b}(X)$. To bound the error term coming from $N_{a,b}(X)$, in Section 4 they analyze toric congruences, most importantly they introduce a result of Pierce, \cite{Pierce1}. The sum $P_{a,b}(Y)$ only contributes to the error term. For this they consider the needed exponential sums in Section 3. Finally in Section 6 the case $a+b\neq 0$ is treated with a case distinction between $ab>0$ and $ab<0$. They conclude in Section 6.3 by choosing \begin{equation}\label{eq: choice X and Y in fkm}
X=q^{9/8-3\beta/4} \text{ and }Y=q^{7/8+3\beta/4}\end{equation} 
for $\beta=\frac{c'}{|a|+|b|}$ for $c'$ fixed and small enough.

Analogous to this we obtain very similar asymptotic formulae if the product of L-functions is twisted by a character. More precisely we need and show the following asymptotics.

\begin{prop}\label{prop: Adaptation of second moment k=0}
Let $a,b\in \ZZ\smallsetminus\{0\}$, $l_1,l_2\leq L$ be positive integers and $q$ a prime number. There exists $\eta=\eta(a,b)$ and some constant $ c(l_1,l_2,b)\geq0$ such that the following holds.
\begin{enumerate}

\item   For $ab>0$ we have 
$$    \frac{1}{q-1}\sumst_{\chi\mods q}\chi(l_1^al_2^b)L(\chi^a,\tfrac 12)L(\chi^b,\tfrac 12)=\begin{cases}
            1 +O(q^{-\eta}) &\text{if }l_1=l_2=1  \\ 
             O(q^{-\eta}) &\text{if $l_1$ or }l_2\neq1.  \\ 
        \end{cases} $$
\item For $a=1$ we have 
$$\frac{1}{q-1}\sumst_{\chi\mods q}\chi(l_1^{-1}l_2^b)L(\chi,\tfrac 12)L(\chi^b,\tfrac 12)=\begin{cases}
              \frac{1}{\sqrt{l_1^{-1}l_2^b}}\zeta\left(\frac{|b|+1}{2}\right) +O(q^{-\eta}) &\text{if }b<-1 \\ 
              c(l_1,l_2,b) +O(q^{-\eta}) &\text{if }b>0.
        \end{cases} $$
  \item  For $a=-1$ we have     
  $$\frac{1}{q-1}\sumst_{\chi\mods q}\chi(l_1l_2^b)L(\chi^{-1},\tfrac 12)L(\chi^b,\tfrac 12)=\begin{cases}
           \frac{1}{\sqrt{l_1l_2^b}} \zeta\left(\frac{|b|+1}{2}\right) +O(q^{-\eta}) &\text{if }b>1 \\ 
             c(l_1,l_2,b) +O(q^{-\eta}) &\text{if }b<0 .\\ 
        \end{cases} $$
    \end{enumerate}
The constant $ c(l_1,l_2,b)$ is the number of solutions to the equation $mn^bl_2^b=l_1$ with $mn\leq q$.
\end{prop}
\begin{rem}
    The constant $c(l_1,l_2,b)$ frequently equals 0. 
\end{rem}

\begin{proof}
We adapt \cite{FKMsecondmom} and henceforth use their notation.
We begin with the treatment of the respective main terms in the case of even characters. First we check that we can still use the results of \emph{loc. cit.} Section 4.

The trivial bound for the case $ ab\neq 0$ given in \cite{FKMsecondmom}, Equation (4.3) does not change upon choosing $u$ appropriately: 
$$\frac{|\{(x,y)\in I\times J :\ x^ay^b\equiv l_1^{\pm a}l_2^{-b}u \mods q\}|}{\sqrt{|I||J|}}\ll \frac{\sqrt{|I|}}{\sqrt{|J|}}\left(\frac{|J|}{q}+1\right)$$
We replace \emph{loc. cit.} Theorem 4.4 by \cite{Nunes_2017}, Lemma 1.4, which gives us that for $q$ prime, $a\in (\ZZ/q\ZZ)^\ast$, $1\leq M\leq q^{3/4}$ and $1\leq N<q/2$ we have
\begin{equation}\label{eq: Nunes replacement for pierce}
    |\{m\sim M, n\sim N; m^u\equiv an^v \mods q\}| \ll M^{2/3}N^{1/4}q^\varepsilon.
\end{equation}
This bound is uniform in $a$. 
Since in \emph{loc. cit.} they use Theorem 4.4 for $k=2$ this bound actually matches theirs. 
    
 Now we adapt \cite{FKMsecondmom}, Section 6.1, where the case $a,b>0$ is treated.
    For (1) we evaluate
    $$    N_{a,b}(X)=\frac 12 \sum_{\substack{m,n\geq 1\\ (m^an^bl_1^al_2^b)^2\equiv 1\mods q}}\frac{1}{\sqrt{mn}}V\left(\frac{mn}{X}\right) .$$
For $l_1=l_2=1$ we have the solution $(m,n)=(1,1)$. If $(l_1,l_2)\neq (1,1)$ we do not have a trivial solution. All contributions of non-trivial solutions are part of the error term. The treatment of these remaining cases is then carried out analogously to \cite{FKMsecondmom} upon choosing $L<q^{\frac{1}{2(a+b)}}$. Then we have $\max (ml_1,nl_2)\geq (q-1)^{\frac{1}{a+b}}$ and hence in particular $\max (m,n)\geq (q-1)^{\frac{1}{2(a+b)}}$ for $l_1,l_2\leq L$. The rest of the argument is analogous to \cite{FKMsecondmom}. 
Hence 
$$\frac{1}{q-1}\sum_{\chi\mods q}\chi(l_1^a)\chi(l_2^b)L(\chi^a,\tfrac 12)L(\chi^b,\tfrac 12)=\begin{cases} 1 + O(q^{-\eta}) & \text{if }l_1=l_2=1 \\ 
       O(q^{-\eta}) &\text{else} \end{cases}.$$

In (2) for $b>0$ (and analogously (3) for $b<0$) we have a main term from the summation over $\chi$ coming from 
$$N_{-1,b}(X)=\frac12\sum_{\substack{m,n\geq 1\\ (\bar m n^bl_1l_2^b)^2\equiv 1\mods q}}\frac{1}{\sqrt{mn}}V\left(\frac{mn}{X}\right).$$
More precisely it is coming from the trivial solutions of the form $mn^bl_2^b=l_1$. The number of these solutions is finite and depends on the values of $l_1,l_2$ and $b$. The treatment of the non-trivial solutions is carried out in the same manner as for (1).

We now turn to \cite{FKMsecondmom}, Section 6.2 where the case $a>0, b<0$ for $a+b\neq 0$ is treated. 
This applies to (2) for $b<-1$ and analogously (3) for $b>1$. We have a main term from the summation over $\chi$ coming from 
$$N_{-1,b}(X)=\frac12\sum_{\substack{m,n\geq 1\\ (\bar m n^bl_1l_2^b)^2\equiv 1\mods q}}\frac{1}{\sqrt{mn}}V\left(\frac{mn}{X}\right).$$
More precisely it is coming from the solutions of the form $m=n^bl_2^bl_1$, denoted by $N^{syst}_{-1,b}(X)$ in \emph{loc. cit.} This gives us $$\frac 12 \frac{1}{\sqrt{l_1l_2^b}}\zeta\left(\frac{b+1}{2}\right)+O\left(\left(\frac{X}{L^{b+1}}\right)^{-1/2+\varepsilon}\right).$$ 
For $L<q^{\frac{1}{2(b+1)}}$ we have that $n^bl_2^bl_1\geq q/2$ implies $\max(m,n)\geq(q/2)^{\frac{1}{2b}}$. The rest of the argument in Section 6.2 is carried out analogously.

Lastly we adapt \cite{FKMsecondmom}, Section 6.3. The odd parts are handled analogously to the even parts, exactly as done in \cite{FKMsecondmom}.
The error terms in all remaining cases are given by 
$$   P_{a,b}(Y)=\frac{1}{2\sqrt q}\sum_{m,n\geq 1}\frac{1}{\sqrt{mn}}\widetilde{T}_{2a,2b}((m^an^bl_1^{\pm a}l_2^{-b})^2;q)V\left(\frac{mn}{Y}\right),$$
    where 
$$\widetilde{T}_{a,b}(u;q)=\frac{1}{\sqrt q}\sum_{\substack{x,y\in \FF_q^\times \\x^ay^b\equiv u \mods q}}e\left(\frac{x+y}{q}\right).$$ 
The treatment of $\widetilde{T}_{a,b}(u;q)$ is insensitive of the shape of $u$ (but might depend on its value), see \cite{FKMsecondmom}, 
Proposition 3.7. 
Hence we once again proceed exactly as in \cite{FKMsecondmom}, Section 6.3 including their choice of $X$ and $Y$ and conclude the proof after choosing $L=q^\alpha$ small enough.

\end{proof}

\begin{proof}[Proof of Theorem \ref{thm:introduction1}] 
In the generic case we want to show that 
$$M_{a,b,k}(\xi;q)\ll_k q^{-\eta}.$$
By Proposition \ref{prop: application AFE}, we have that $M_{a,b,k}(\xi;q)$ is essentially a linear combination of $M_{a,b,\Kl_{|k|}}(\pm\xi)$ and $M_{a,b,\K_{|k|}^{a,b}}(\pm\xi)$. By Corollary \ref{cor: conclusion chap 4} the latter two can be bounded as desired once we establish \eqref{eq: type 1 eq} and \eqref{eq: type 1 gallant}.
The bound \eqref{eq: type 1 eq} is given by Proposition \ref{prop:ab for e_q}. The bound \eqref{eq: type 1 gallant} follows from Theorem \ref{thm weakly gkr from fkms} for all gallant cases including the cases from Lemma \ref{lem: special cases}.
This concludes the generic case.

It remains to show that
    \begin{align*}
        \frac{1}{q-1}&\sumst_{\chi \mods{q}}\chi(l_1^al_2^b)L\left(\chi^a,\tfrac{1}{2}\right)L\left(\chi^b,\tfrac{1}{2}\right)\varepsilon(\chi)^k \\  
     =&\begin{cases} \frac{\zeta(\frac{b+1}{2})}{l_1^{-1/2}l_2^{b/2}}+ O_k( q^{-\eta}) &\text{if }(a,b,k)=(1,b>1,-1) \text{ or } (-1,b<-1,1) \\
          c(l_1,l_2,b)+ O_k( q^{-\eta}) &\text{if }(a,b,k)=(1,b<0,-1) \text{ or } (-1,b>0,1). \\
        \end{cases}
    \end{align*}
For $a=1$ and $k=-1$ (or analogously for $a=-1$ and $k=1$) we have by the functional equation that 
\begin{align*}
    \sumst_{\chi \mods q}\chi(l_1l_2^b)L(\chi,\tfrac 12)L(\chi^b,\tfrac 12)\overline{\varepsilon( \chi)}=\sumst_{ \chi \mods q}i^{\kappa(\chi)}\chi(-l_1l_2^b)L(\bar \chi,\tfrac 12)L(\chi^b,\tfrac 12) ,\\
\end{align*}
where $\kappa(\chi)=1$ if $\chi$ is odd and 0 otherwise.  Now the claim follows from Proposition \ref{prop: Adaptation of second moment k=0} upon noting that even and odd characters are treated separately and hence the factor $\chi(-1)i^{\kappa(\chi)}$ does not change the asymptotics. 
\end{proof}

\section{Mollifiers and non-vanishing}\label{s: non-vanishing}

In this section, we prove the following technical result, which proves Theorem \ref{thm:intro nonvanishing}.

\begin{thm}\label{thm:non-vanishing}Let $I \subset (-\pi,\pi]$ be a non-empty interval and let $\mu(I)$ denote the Lebesgue measure of $I$. Let $a,b$ be two non-zero integers. 
Then, there exist constants $c_2 \neq 0\neq c_4$, $c_2'$, $\eta'>0$ and $c$ such that for every $\epsilon >0$ there is a constant $C(\epsilon)>0$ such that
    \begin{equation*}\label{eq:intrononvanishing2}
   \frac{1}{q-1}E(a,b;q,I) \geq \frac{1}{c_4^2}\left(\mu(I)c_2-\epsilon\vert{ab}\vert^{1/2}c_2'-\frac{c}{\log q}-C(\epsilon)q^{-\eta'}\right)^2.
\end{equation*}
Moreover $c = 0$ if $\pm 1\notin \{a,b\}$.
In particular, for every $q$ large enough, there is a positive proportion of Dirichlet characters whose angle lies in $I$ so that both central $L$-values $L\left(\chi^a,\tfrac{1}{2}\right),$ $L\left(\chi^b,\tfrac{1}{2}\right)$ are non-vanishing. 
\end{thm}

\begin{rem}\label{rem:side effect mollification technique}
    We observe the following beneficial side effect of the mollification technique. Without using it, we get an extra factor $\log q^A$, for some $A>0$, multiplied by the term $\epsilon\vert{ab}\vert^{1/2}$ in \eqref{eq:intrononvanishing}, which spoils even the $0$-proportion result.
\end{rem}

\begin{rem}
    The constants $c_2,c_2',c_4$, $c$ in Theorem \ref{thm:non-vanishing}, can be explicitly computed, for $c_2$ see Propositions \ref{prop:secondmollifiedmomentk=0}, for $c_4$ see Proposition \ref{prop:fourthmollifiedmomenet} and for $c_2'$ see the proof of the Theorem. For the constant $c$, see Proposition \ref{prop:secondmollifiedmomentkneq0}, note that $c_{2,-1}$ in this last statement can be made explicit in terms of $a,b$ and some value of the Riemann zeta function. 
\end{rem}

Since $L(\chi,\tfrac 12)$ vanishes if and only if $L(\bar\chi,\tfrac 12)$ does, we restrict ourselves to $a,b>0$. To complete the proof of Theorem \ref{thm:non-vanishing}, we start with the evaluations of the mollified fourth moment and the mollified second moment for $k=0$ and $k\neq 0$. Here the case $k=-1$ again requires special treatment because of the functional equation.

First, we introduce the mollified moments. Historically, mollifiers have been introduced by Selberg, \cite{Selberg1942}.
Let $q^{o(1)}\leq L=q^\alpha< q^{1/2}$ be a parameter to be chosen precisely later. For $\chi \in \widehat{(\ZZ/q\ZZ)^{\times}}$ we define the mollifier 
\[M_L(\chi) = \sum_{l\leq L}\frac{\mu(l)\chi(l)}{l^{1/2}}P\left(\frac{\log (L/l)}{\log L}\right),\]
where $P\in \RR[X]$ is some polynomial (of absolutely bounded degree), so that $P(0) = 0$ and $P(1) = 1$.
Applying the triangle inequality yields \begin{equation}\label{eq: trivial bound mollifier}
    \vert{M_L(\chi)}\vert \ll L^{1/2} \sup_{x\in [0,1]}\vert{P(x)}\vert.
\end{equation}
We furthermore need the following bound. 

\begin{lem}\label{lem: bound for a=1, with mollifier} Let $b\geq 1$ be an integer and $L>2$. Then we have the following.
\[\sum_{1\leq l_1,l_2\leq L}\frac{\mu(l_1)\mu(l_2)}{l_1l_2^{\frac{b+1}{2}}}P_L\left(\frac{\log(L/l_1)}{\log L}\right)P_L\left(\frac{\log(L/l_2)}{\log L}\right) \ll \frac{1}{\log L} \] 
\end{lem}
\begin{proof}
    After opening the polynomial $P_L$ we see that it is sufficient to obtain the following bounds for $j\geq 0$ and $r>1$:
    \begin{equation*}
        \sum_{1\leq l\leq L}\frac{\mu(l)(\log l)^j}{l^{\frac{r+1}{2}}} \ll_j 1, \qquad \sum_{1\leq l\leq L}\frac{\mu(l)\log(l)^{j+1}}{l} \ll_j 1, \qquad \sum_{1\leq l\leq L}\frac{\mu(l)}{l} \ll \frac{1}{\log L}.
    \end{equation*}
    By a version of the Prime number theorem (see  \cite{Montgomery_Vaughan_2006}, (6.18)), there exists a constant $C>0$ such that 
    \[\sum_{1\leq l\leq L}\frac{\mu(l)}{l} \ll e^{-C\sqrt{\log L}} \ll \frac{1}{\log L}.\]
    Also, since $r>1$, one has 
    \[\sum_{1\leq l\leq L}\frac{\mu(l)(\log l)^j}{l^{\frac{r+1}{2}}} = \left(\frac{1}{\zeta}\right)^{(j)}\left(\frac{r+1}{2}\right)- \sum_{l>L}\frac{\mu(l)(\log l)^j}{l^{\frac{r+1}{2}}} \]
    and the tail is bounded by 
    \[\sum_{l>L}\frac{\mu(l)(\log l)^j}{l^{\frac{r+1}{2}}} \ll \int_{L}^{\infty}\frac{(\log t)^j}{t^{\frac{r+1}{2}}} \ \mathrm{d}t \ll_j (\log L)^jL^{-\frac{r-1}{2}}.\]
    Finally, we have by summation by parts,
    \begin{align*}\sum_{1\leq l\leq L}\frac{\mu(l)\log(l)^{j+1}}{l} =& -j\int_{1}^{\infty}\left(\sum_{1\leq l\leq t}\frac{\mu(l)}{l}\right)\frac{(\log t)^j}{t} \ \mathrm{d}t  \\ & + (\log L)^{j+1}\sum_{1\leq l\leq L}\frac{\mu(l)}{l} + j\int_{L}^{\infty}\left(\sum_{1\leq l\leq t}\frac{\mu(l)}{l}\right)\frac{(\log t)^j}{t} \ \mathrm{d}t.\end{align*}
    Again, the tail is bounded by \[(\log L)^{j+1}\sum_{1\leq l\leq L}\frac{\mu(l)}{l} + j\int_{L}^{\infty}\left(\sum_{1\leq l\leq t}\frac{\mu(l)}{l}\right)\frac{(\log t)^j}{t} \ \mathrm{d}t\ll_j (\log L)^{j+\frac{1}{2}}e^{-C\sqrt{\log L}}.\] 
\end{proof}

We define the mollified second moment 
\begin{equation}\label{def:mollifiedsecondmoment}
    \mathcal{Q}_2(a,b,k,M_L)= \frac{1}{q-1}\sumst_{\chi\mods{q}}L\left(\chi^a,\tfrac{1}{2}\right)L\left(\chi^b,\tfrac{1}{2}\right)M_L(\chi^a)M_L(\chi^b)\varepsilon(\chi)^k
\end{equation}
and the mollified fourth moment 
\begin{equation*}
    \mathcal{Q}_4(a,b,M_L)=\frac{1}{q-1}\sum_{\chi \mods{q}}\left\vert{L\left(\chi^a,\tfrac{1}{2}\right)L\left(\chi^b,\tfrac{1}{2}\right)M_L(\chi^a)M_L(\chi^b)}\right\vert^2.
\end{equation*}
 From Iwaniec and Sarnak, \cite{IwaniecSarnak+1999+941+952}, one deduces that for $P(X) = X$ one has
\begin{equation}\label{eq:mollifiedsecondmomentclassical}\frac{1}{q-1}\sum_{\chi\mods{q}}\vert{M_L(\chi)L(\chi,\tfrac 12)}\vert^2 \ll_\alpha 1.\end{equation}
The same bound holds for $P(X)=X^2$ (for example using \eqref{eq: bounds from zacharias} and Cauchy-Schwarz).
\begin{prop}[Evaluation of second mollified moment, $k=0$]\label{prop:secondmollifiedmomentk=0}For any positive integers $a,b$ there exist $c_2$ and $\eta >0$ so that 
    \begin{align*}
        \mathcal{Q}_2(a,b,0,M_L) &= P(1)^2 + O(Lq^{-\eta}).
         \end{align*}
\end{prop}

\begin{proof}After opening the mollifiers we write
   \begin{align*}
         & \mathcal{Q}_2(a,b,0,M_L) \\&=\frac{1}{q-1}  \sum_{l_1,l_2\leq L}\frac{\mu(l_1)\mu(l_2)}{(l_1l_2)^{1/2}}P\left(\frac{\log L/l_1}{\log L}\right)P\left(\frac{\log L/l_2}{\log L}\right)\sumst_{\chi\mods q}\chi(l_1^a)\chi(l_2^b)L(\chi^a,\tfrac 12)L(\chi^b,\tfrac 12).
    \end{align*}
By Proposition \ref{prop: Adaptation of second moment k=0} we have
$$\frac{1}{q-1}\sum_{\chi\mods q}\chi(l_1^a)\chi(l_2^b)L(\chi^a,\tfrac 12)L(\chi^b,\tfrac 12)=\begin{cases} 1 + O(q^{-\eta}) & \text{if }l_1=l_2=1 \\ 
       O(q^{-\eta}) &\text{else.} \end{cases}$$
Removing the contribution from the trivial character and using \eqref{eq: trivial bound mollifier} we conclude the proposition. 
\end{proof}

\begin{prop}[Evaluation of second mollified moment, $k\neq 0$]\label{prop:secondmollifiedmomentkneq0}For positive integers $a\leq b$ not both equal to 1 and any integer $k\neq 0$, there exist constants $c_{2,k} = c_{2,k}(a,b,P)>0$ and $\eta >0$ so that 
    \begin{align*}
        \mathcal{Q}_2(a,b,k,M_L) &\leq c_{2,k}Lq^{-\eta} &\text{if }(a,b,k)\neq (1,b,-1) \\
        \mathcal{Q}_2(a,b,k,M_L) &\leq c_{2,-1}\frac{1}{\log L} &\text{if }(a,b,k)=(1,b,-1).
         \end{align*}
\end{prop}

\begin{proof}
    We again open the mollifiers and obtain 
    \begin{align*}
          \mathcal{Q}_2(a,b,k,M_L) &= \sum_{l_1,l_2\leq L}\frac{\mu(l_1)\mu(l_2)}{(l_1l_2)^{1/2}}P\left(\frac{\log L/l_1}{\log L}\right)P\left(\frac{\log L/l_2}{\log L}\right)M_{a,b,k}(l_1^al_2^b;q). 
    \end{align*}
   For $a,b\neq 1$, 
   since $k\neq 0$, we have by Theorem \ref{thm:introduction1} that $M_{a,b,k}(l_1^al_2^b;q) = O(q^{-\eta})$ for some $\eta >0$ independent of $k$, but possibly depending on $a,b$. 
In the second case we have by Theorem \ref{thm:introduction1}, with the error term as in the proof of Proposition \ref{prop: Adaptation of second moment k=0}, that $M_{1,b,-1}(l_1l_2^b;q)$ is estimated as $$\frac 12 \frac{1}{\sqrt{l_1l_2^b}}\zeta\left(\frac{b+1}{2}\right)+O\left(\left(\frac{X}{L^{b+1}}\right)^{-1/2+\varepsilon}\right).$$
Plugging this into our original question (with $P(X)=X^2$) we are led to investigate 
\begin{align*}
\zeta\left(\frac{b+1}{2}\right) \sum_{l_1,l_2\leq L}&\frac{\mu(l_1)\mu(l_2)}{l_1l_2^{\frac{b+1}{2}}}P\left(\frac{\log L/l_1}{\log L}\right)P\left(\frac{\log L/l_2}{\log L}\right)\\ &+O\left(\frac{L^{\frac{b+1}{2}}}{X^{1/2}}\sum_{l_1,l_2\leq L}\frac{\mu(l_1)\mu(l_2)}{l_1l_2^{\frac{b+1}{2}}} P\left(\frac{\log L/l_1}{\log L}\right)P\left(\frac{\log L/l_2}{\log L}\right)\right). 
\end{align*}
Using \eqref{eq: trivial bound mollifier} and $P(x)=x^2$ the error term is bounded by 
$ \ll \frac{L^{\frac{b+3}{2}}}{X^{1/2-\varepsilon}}.$ For $L=q^\alpha$ and $X$ as in \eqref{eq: choice X and Y in fkm}, choosing $\alpha$ and $\beta$ small enough we see that this is bounded by $q^{-\eta}$ for some $\eta>0$.
Finally, by Lemma \ref{lem: bound for a=1, with mollifier} we conclude that the main term is bounded by $\ll \frac{1}{\log L }$. 
\end{proof}

\begin{prop}[Bounds for the fourth mollified moment]\label{prop:fourthmollifiedmomenet}For any non-zero integers $a, b$ and $P(X)=X^2$ there exists $c_4 >0$ (independent of $q$) so that 
     \begin{align*}
              \mathcal{Q}_4(a,b,M_L) \leq  c_4.
         \end{align*}
\end{prop}
\begin{proof}
    By \cite{Zacharias2016MollificationOT}, Theorem 1.2 we know that for $P(X)=X^2$ and any $0<\lambda<\frac{11}{8064}$ we have
    \begin{equation}\label{eq: bounds from zacharias}
     \frac{1}{q-1}\sumst_{\chi\mods q}|L\left(\chi,\tfrac 12\right)M_L(\chi)|^4=\sum_{i=0}^4d_i\lambda^{-i}+O_\lambda\left(\frac{1}{\log q}\right)\end{equation}
    for some calculable coefficients $d_i\in \RR$.
    Using Cauchy-Schwarz we obtain 
    \begin{align*}
    \mathcal{Q}_4(a,b,M_L)&\leq \left(\frac{1}{q-1}\sumst_{\chi\mods q}|L\left(\chi^a,\tfrac 12\right)M_L(\chi^a)|^4\right)^{\frac12}\left(\frac{1}{q-1}\sumst_{\chi\mods q}|L\left(\chi^b,\tfrac 12\right)M_L(\chi^b)|^4\right)^{\frac12} \\
    & \leq (ab)^{1/2}\left(\sum_{i=0}^4d_i\lambda^{-i}+O_\lambda\left(\frac{1}{\log q}\right)\right)  \\
    &\leq c_4.\qedhere\end{align*} 
\end{proof}

\begin{proof}[Proof of Theorem \ref{thm:non-vanishing}]

Let $I \subset (-\pi,\pi]$ be a non-empty interval. Let $a,b$ be positive integers.
We are interested in lower bounds for the following quantity. \[E(a,b;q,I)= \vert{ \{\chi  \mods q\ \text{ non-trivial s.t. }  L(\chi^a,\tfrac 12)L(\chi^b,\tfrac 12)\neq  0,\  \theta(\chi) \in I  \}}\vert\] 
    Let $\phi \in C_c^{\infty}(-\pi,\pi)$ be such that $0\leq \phi\leq 1_I$ and \[\int_{-\pi}^{\pi}\phi(\theta)\ \mathrm{d}\theta \geq \frac{\mu(I)}{2}.\] Let $\epsilon >0$ and let $K = K(\epsilon)$ be so that $\sum_{\vert{k}\vert \geq K}\vert{\widehat{\phi}(k)}\vert < \epsilon,$
    where $\widehat{\phi}(k)$ denotes the $k$-th Fourier coefficient of $\phi$. 
     We have
     \begin{align*}
          &\left\vert{\sum_{\vert{k}\vert >K}\widehat{\phi}(k)\mathcal{Q}_2(a,b,k,L;q)}\right\vert \\&\qquad\qquad= \frac{1}{q-1}\left\vert{\sumst_{\chi \mods{q}}L\left(\chi^a,\tfrac{1}{2}\right)L\left(\chi^b,\tfrac{1}{2}\right)M_L(\chi^a)M_L(\chi^b)\sum_{\vert{k}\vert >K}\widehat{\phi}(k)\varepsilon(\chi)^k}\right\vert\\ &\qquad\qquad\leq \epsilon \left(\sum_{\chi\mods{q}}\vert{M_L(\chi^a)L(\chi^a,\frac{1}{2})}\vert^2\right)^{1/2}\left(\sum_{\chi\mods{q}}\vert{M_L(\chi^b)L(\chi^b,\tfrac{1}{2})}\vert^2\right)^{1/2}\\ &\qquad\qquad\leq \epsilon (ab)^{1/2}\sum_{\chi\mods q}\vert{M_L(\chi)L(\chi,\tfrac{1}{2})}\vert^2\leq \epsilon({ab})^{1/2}c'_2.
    \end{align*}
     From Proposition \ref{prop:secondmollifiedmomentkneq0} we have the following:

    \begin{align*}
    \vert{\sum_{0< \vert{k}\vert \leq K(\epsilon)}\widehat{\phi}(k)\mathcal{Q}_2(a,b,k,M_L)}\vert &\leq Lq^{-\eta} \sum_{\substack{1\leq \vert{k}\vert \leq K(\epsilon)\\ k\neq -1}}c_{2,k}\vert{\widehat{\phi}(k)}\vert + c_{2,-1}(\log L)^{-1} \\ &\leq Lq^{-\eta}C(\epsilon) + \frac{c_{2,-1}\widehat{\phi}(-1)}{\log L}   
    \end{align*}
     Note that the second summand $\frac{c_{2,-1}\widehat{\phi}(1)}{\log L}$ appears only if $1\in \{a,b\}$.
    Hence, using Proposition \ref{prop:secondmollifiedmomentk=0} and Proposition \ref{prop:fourthmollifiedmomenet} we obtain

    \begin{align*}
        \widehat{\phi}(0)c_2-\epsilon (ab)^{1/2}c'_2-\frac{c_{2,-1}\widehat{\phi}(-1)}{\log L} &- Lq^{-\eta}\sum_{1\leq \vert{k}\vert \leq K}c_{2,k}\vert{\widehat{\phi}(k)}\vert  \leq \vert{\sum_{k\in \ZZ}\widehat{\phi}(k)\mathcal{Q}_2(a,b,k,M_L)}\vert \\& \leq \left(\frac{1}{q-1}\sum_{\substack{\chi\mods{q}\\L(\chi^a,\tfrac{1}{2})\neq 0\neq L(\chi^b,\tfrac{1}{2})}}\vert{\phi(\theta(\chi))}\vert^2\right)^{1/2}\mathcal{Q}_4(a,b,M_L) \\ &\leq \left(\frac{E(a,b;q,I)}{q-1}\right)^{1/2} c_4.
    \end{align*}
For $(a,b)=(1,1)$ we do not have an asymptotic formula, but we deduce the same non-vanishing result for $L(\chi,\tfrac 12)L(\chi,\tfrac 12)$ from the non-vanishing for $L(\chi,\tfrac 12)L(\chi^b,\tfrac 12)$ with $ b>1$. 
   \end{proof}

\printbibliography 
\nocite{Blomeretal}

\end{document}